\documentclass[11pt]{article}

\usepackage{amssymb,amsmath,amsfonts,amsthm}
\usepackage{latexsym}
\usepackage{graphics}
\usepackage{indentfirst}
\usepackage{lmodern}
\usepackage{hyperref}
\usepackage{ mathrsfs }
\usepackage{array}
\usepackage{mathtools}
\usepackage{graphicx}
\usepackage{setspace}
\usepackage{enumerate}
\usepackage{tikz}
\usepackage{url,color}
\usepackage{chngcntr}
\usepackage{multicol}

\usepackage{comment}

\newtheorem{innercustomthm}{Theorem}
\newenvironment{customthm}[1]
  {\renewcommand\theinnercustomthm{#1}\innercustomthm}
  {\endinnercustomthm}

\newtheorem{innercustomcor}{Corollary}
\newenvironment{customcor}[1]
  {\renewcommand\theinnercustomcor{#1}\innercustomcor}
  {\endinnercustomcor}

\newtheorem{innercustomconj}{Conjecture}
\newenvironment{customconj}[1]
  {\renewcommand\theinnercustomconj{#1}\innercustomconj}
  {\endinnercustomconj}
  
\newtheorem*{thm*}{Theorem}
\newtheorem{thm}{Theorem}
\newtheorem{lem}[thm]{Lemma}
\newtheorem{pro}[thm]{Proposition}
\newtheorem{obs}[thm]{Observation}
\newtheorem{rmk}[thm]{Remark}
\newtheorem{cor}[thm]{Corollary}
\newtheorem{conj}[thm]{Conjecture}

\newtheorem{fact}[thm]{Fact}

\newcommand{\N}{\mathbb{N}}
\newcommand{\F}{\mathbb{F}}

\newcommand{\Z}{\mathbb{Z}}

\newcommand{\R}{\mathbb{R}}

\definecolor{vividviolet}{rgb}{0.62, 0.0, 1.0}
\definecolor{OliveGreen}{rgb}{0.0, .6, .1}

\newcommand{\tb}{\textcolor{black}}

\begin{document}

\title{A Polynomial Method for Counting Colorings of Sparse Graphs
}

\author{Samantha L. Dahlberg$^1$, Hemanshu Kaul$^2$, and Jeffrey A. Mudrock$^3$}

\footnotetext[1]{Department of Applied Mathematics, Illinois Institute of Technology, Chicago, IL 60616.  E-mail:  {\tt {sdahlberg@iit.edu}}}
\footnotetext[2]{Department of Applied Mathematics, Illinois Institute of Technology, Chicago, IL 60616.  E-mail:  {\tt {kaul@iit.edu}}}
\footnotetext[3]{Department of Mathematics and Statistics, University of South Alabama, Mobile, AL 36688.  E-mail:  {\tt {mudrock@southalabama.edu}}}

\maketitle

\begin{abstract}
The notion of $S$-labeling of graphs, where $S$ is a subset of a symmetric group, was introduced in 2019 by Jin, Wong, and Zhu.  This notion provides the framework for a common generalization of various well studied notions of graph coloring, including classical coloring, signed $k$-coloring, signed $\mathbb{Z}_k$-coloring, DP (or correspondence) coloring, group coloring, and coloring of gained graphs.  In this paper, we present a unified and simple polynomial method for giving exponential lower bounds on the number of colorings of an $S$-labeled graph for all such $S$. This algebraic technique allows us to prove new lower bounds on the number of colorings of any $S$-labeling of graphs satisfying certain sparsity conditions. We also investigate how the structure of $S$ can be exploited to improve the applicability of these bounds. Our results give new lower bounds on the number of DP-colorings, and consequently the number of all types of colorings listed above.  This includes the chromatic polynomial and the number of list colorings of families of planar graphs, and the number of colorings of signed graphs. These enumerative bounds improve previously known results or are the first such known results.

\medskip
  


\end{abstract}

\section{Introduction}\label{intro}

The coloring of planar graphs and its subfamilies has a long history starting with the four color conjecture in the nineteenth century. This history is intertwined with the related conjectures on existence of exponentially many such colorings (as a function of the number of vertices), going back at the least to Birkhoff's work in 1930~\cite{B30}, and Birkhoff and Lewis' enumerative extension of the four color conjecture in 1946~\cite{BL46} that for any $n$-vertex planar graph $G$, the chromatic polynomial satisfies $P(G,k) \ge k(k - 1)(k-2)(k-3)^{n-3}$ for all real numbers $k \ge 4$. They proved this is true for $k\ge 5$, thus giving exponentially many 5-colorings of planar graphs. After Thomassen~\cite{T94} in 1994 famously proved all planar graphs are 5-choosable, there has been much work done on showing that planar graphs and their subfamilies have exponentially many list $k$-colorings for appropriate $k\in \{3,4,5\}$~\cite{ADPT13,BG22, DKM22, DP22, LT21, PS23, T06, T07, T07b, T23}. These proofs are typically intricate topological arguments specialized to the family of planar graphs under consideration. 

In this paper, we give a unified and simple polynomial method~(\cite{T14}) for proving many such exponential lower bounds on the number of colorings of sparse graphs without requiring any topological assumptions. Our results are set in the general framework of coloring $S$-labeled graphs, where $S$ is a subset of a symmetric group, which includes classical and many other well-studied notions of graph coloring, such as signed $k$-coloring, signed $\mathbb{Z}_k$-coloring, DP (or correspondence) coloring, group coloring, and coloring of gained graphs,  as special cases. In fact, for each such choice of $S$ there is a corresponding notion of coloring of a graph. The subset relation on the set of nonempty subsets of the symmetric group, induces a partial order (in fact, a distributive lattice) on all these notions of coloring with the so-called DP coloring as the unique maximal element and the classical coloring as a minimal element (see Remark~\ref{rmk:poset}). 

Our results give an unified algebraic perspective and toolbox for proving lower bounds on the enumerative functions of \textit{all} these notions of coloring, as well as list coloring, for any appropriate sparse graph. We also investigate how the structure of $S$ can be exploited to improve the applicability of these bounds. Application of our lower bounds to families of planar graphs either improve previously known enumerative results or are the first such known results. Many of these families of planar graphs have previously been studied in the literature but only from the perspective of their various chromatic numbers. There are no previously known lower bounds, much less exponential lower bounds, on the enumerative functions of any notion of coloring for these classes of planar graphs that are considered in our applications. For example, Signed graphs and their colorings have been widely studied since 1980s (see a recent survey in \cite{SV21}), but we do not know of any previous results on bounding the number of such colorings.

In the rest of this section, we formally define these concepts and give an outline of our results.

\subsection{Graph Coloring and Counting Colorings}

 We begin by establishing some terminology, notation, and conventions that will be used throughout this paper.  All graphs are nonempty, finite, undirected loopless multigraphs or nonempty, finite, simple digraphs.  For the purposes of this paper, a simple graph is a multigraph without any parallel edges between vertices.  Generally speaking we follow West~\cite{W01} for terminology and notation.  The set of natural numbers is $\N = \{1,2,3, \ldots \}$.  For $m \in \N$, we write $[m]$ for the set $\{1, \ldots, m \}$.  Suppose $G$ is a multigraph.  When $u,v \in V(G)$ we use $E_G(u,v)$ to denote the set of edges in $E(G)$ with endpoints $u$ and $v$ (note $E_G(u,v) = E_G(v,u)$), and we let $e_G(u,v)$ denote the number of elements in $E_G(u,v)$.  We let $\mu(G)$ be the maximum value of $e_G(u,v)$ taken over all pairs of vertices in $G$.  Also, the \emph{underlying graph of $G$}, $G'$, is the simple graph formed from $G$ by deleting edges so that $e_{G'}(u,v)=1$ whenever $u$ and $v$ are adjacent in $G$.  When $G$ is a simple graph, we can refer to edges by their endpoints; for example, if $u$ and $v$ are adjacent in the simple graph $G$, $uv$ or $vu$ refers to the edge between $u$ and $v$.  When $D$ is a digraph we refer to edges as ordered pairs.

In classical vertex coloring one wishes to color the vertices of a graph $G$ with up to $k$ colors from $[k]$ so that adjacent vertices receive different colors, a so-called \emph{proper $k$-coloring}.  The \emph{chromatic number} of a graph, denoted $\chi(G)$, is the smallest $k$ such that $G$ has a proper $k$-coloring.  We say $G$ is \emph{$k$-colorable} whenever $k \geq \chi(G)$.  List coloring is a generalization of classical vertex coloring introduced independently by Vizing~\cite{V76} and Erd\H{o}s, Rubin, and Taylor~\cite{ET79} in the 1970s.  In list coloring, we associate a \emph{list assignment} $L$ with a graph $G$ so that each vertex $v \in V(G)$ is assigned a list of available colors $L(v)$ (we say $L$ is a list assignment for $G$).  We say $G$ is \emph{$L$-colorable} if there is a proper coloring $f$ of $G$ such that $f(v) \in L(v)$ for each $v \in V(G)$ (we refer to $f$ as a \emph{proper $L$-coloring} of $G$).  A list assignment $L$ is called a \emph{$k$-assignment} for $G$ if $|L(v)|=k$ for each $v \in V(G)$.  The \emph{list chromatic number} of a graph $G$, denoted $\chi_\ell(G)$, is the smallest $k$ such that $G$ is $L$-colorable whenever $L$ is a $k$-assignment for $G$.

In 1912 Birkhoff~\cite{B12} introduced the notion of the chromatic polynomial with the hope of using it to make progress on the four color problem.  For $k \in \N$, the \emph{chromatic polynomial} of a graph $G$, $P(G,k)$, is the number of proper $k$-colorings of $G$. It is well-known that $P(G,k)$ is a polynomial in $k$ of degree $|V(G)|$.  The notion of chromatic polynomial was extended to list coloring in the early 1990s by Kostochka and Sidorenko~\cite{AS90}.  If $L$ is a list assignment for $G$, we use $P(G,L)$ to denote the number of proper $L$-colorings of $G$. The \emph{list color function} of $G$, denoted $P_\ell(G,k)$, is the minimum value of $P(G,L)$ where the minimum is taken over all possible $k$-assignments $L$ for $G$.  Since a $k$-assignment could assign the same $k$ colors to every vertex in a graph, it is clear that $P_\ell(G,k) \leq P(G,k)$ for each $k \in \N$.  In general, the list color function can differ significantly from the chromatic polynomial for small values of $k$.  One reason for this is that a graph can have a list chromatic number that is much higher than its chromatic number.   On the other hand,  in 1992, Donner~\cite{D92} showed that for any graph $G$ there is a $m \in \N$ such that $P_\ell(G,k) = P(G,k)$ whenever $k \geq m$.  This was recently improved by Dong and Zhang~\cite{DZ22} who showed that $P_\ell(G,k) = P(G,k)$ whenever $k \geq |E(G)|-1$.  

One of our main interests in this paper is a generalization of list coloring called DP-coloring.  In 2015, Dvo\v{r}\'{a}k and Postle~\cite{DP15} introduced DP-coloring (they called it correspondence coloring) in order to prove that every planar graph without cycles of lengths 4 to 8 is 3-choosable.  DP-coloring has been extensively studied over the past 8 years (see e.g.,~\cite{B17, HK21, KM19, KM20, KM21, M18, MT20}). We will introduce DP-coloring via the notion of $S$-labeling.
 
\subsection{Coloring S-labeled Graphs and DP-Coloring}

The notion of $S$-labeling was introduced in 2019~\cite{JW19} as a common generalization of coloring of signed graphs, signed $\mathbb{Z}_k$-coloring, DP-coloring, group coloring, coloring of gained graphs and more (see Remark~\ref{rmk:poset} below). Suppose that $S$ is some nonempty subset of the symmetric group $S_A$ over some set $A$. We will often refer to the elements of $A$ as \emph{colors}.  We use $I_A$ to denote the subset of $S_A$ that contains only the identity permutation.  We will use $S_k$ (resp., $I_k$) to denote $S_{\mathbb{Z}_k}$ (resp., $I_{\mathbb{Z}_k}$). An \emph{S-labeling} of a multigraph $G$ is a pair $(D, \sigma)$ consisting of an orientation $D$ of the underlying graph of $G$ and a function $\sigma : E(D) \rightarrow \bigcup_{i=1}^{\mu(G)} S^i$   with the property $\sigma(u,v) \in S^{e_G(u,v)}$ for each $(u,v) \in E(D)$~\footnote{In the case that $G$ is a simple graph, we take the codomain of $\sigma$ to be $S$.}. Essentially $\sigma$ assigns a $e_G(u,v)$-tuple of permutations to $(u,v)$, that is a permutation to each edge of the multigraph.  A pictorial representation of an $S$-labeling when $G=P_3$ and $S = S_4$ is in Figure~\ref{fig:cover_visual}.

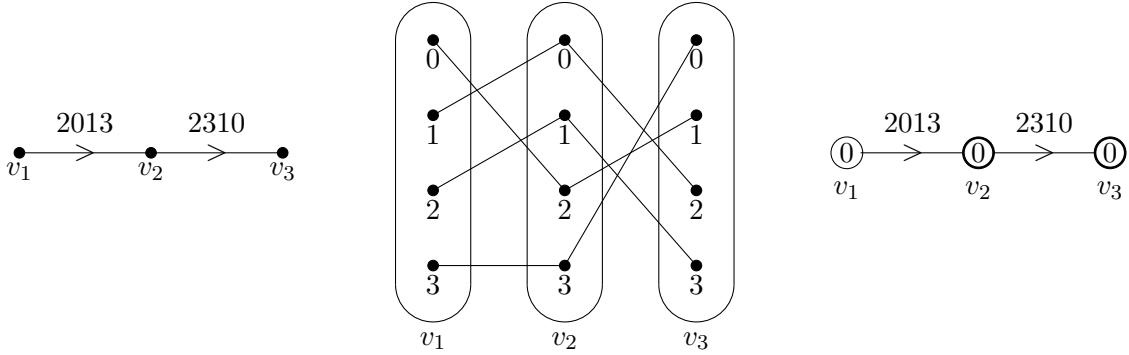
\begin{figure}[t]
\begin{center}
\begin{tikzpicture}
\coordinate (A) at (0,0);
\coordinate (B) at (1.75,0);
\coordinate (C) at (3.5,0);
\draw (A)--(B)--(C);
\filldraw[black] (A) circle [radius=2pt] node[below] {$v_1$};
\filldraw[black] (B) circle [radius=2pt] node[below] {$v_2$};
\filldraw[black] (C) circle [radius=2pt] node[below] {$v_3$};
\draw (.875,0) node {$>$};
\draw (.875,.4) node {$2013$};
\draw (2.625,0) node {$>$};
\draw (2.625,.4) node {$2310$};
%%%%%%%%%%%%%%%%%%%%%%%%%%%%
\begin{scope}[shift={(11,0)}]
\coordinate (A) at (0,0);
\coordinate (B) at (1.75,0);
\coordinate (C) at (3.5,0);
\draw (A)--(B)--(C);
\draw (.875,0) node {$>$};
\draw (.875,.4) node {$2013$};
\draw (2.625,0) node {$>$};
\draw (2.625,.4) node {$2310$};
\draw[black] (A) circle [radius=6pt];
\filldraw[black] (B) circle [radius=6pt];
\filldraw[black] (C) circle [radius=6pt];

\draw (A) + (0,-.5) node {$v_1$};
\draw (B) + (0,-.5) node {$v_2$};
\draw (C) + (0,-.5) node {$v_3$};

\filldraw[white] (A) circle [radius=5pt] node {\tb{$0$}};
\filldraw[white] (B) circle [radius=5pt] node {\tb{$0$}};
\filldraw[white] (C) circle [radius=5pt] node {\tb{$0$}};
\end{scope}
%%%%%%%%%%%%%%%%%%%%%%%%%%%%
\begin{scope}[shift={(5.5,2.5)}]
\coordinate (A) at (0,-1);
\coordinate (B) at (0,-2);
\coordinate (C) at (0,-3);
\coordinate (D) at (0,-4);
\coordinate (AA) at (1.75,-1);
\coordinate (BB) at (1.75,-2);
\coordinate (CC) at (1.75,-3);
\coordinate (DD) at (1.75,-4);
\coordinate (AAA) at (3.5,-1);
\coordinate (BBB) at (3.5,-2);
\coordinate (CCC) at (3.5,-3);
\coordinate (DDD) at (3.5,-4);
\draw (A)--(CC)--(BBB);
\draw (B)--(AA)--(CCC);
\draw (C)--(BB)--(DDD);
\draw (D)--(DD)--(AAA);
\filldraw[black] (A) circle [radius=2pt] node[below] {$0$};
\filldraw[black] (B) circle [radius=2pt] node[below] {$1$};
\filldraw[black] (C) circle [radius=2pt] node[below] {$2$};
\filldraw[black] (D) circle [radius=2pt] node[below] {$3$};
\filldraw[black] (AA) circle [radius=2pt] node[below] {$0$};
\filldraw[black] (BB) circle [radius=2pt] node[below] {$1$};
\filldraw[black] (CC) circle [radius=2pt] node[below] {$2$};
\filldraw[black] (DD) circle [radius=2pt] node[below] {$3$};
\filldraw[black] (AAA) circle [radius=2pt] node[below] {$0$};
\filldraw[black] (BBB) circle [radius=2pt] node[below] {$1$};
\filldraw[black] (CCC) circle [radius=2pt] node[below] {$2$};
\filldraw[black] (DDD) circle [radius=2pt] node[below] {$3$};
\draw (0.5,-1)  arc(0:180:.5);
\draw (-.5,-1)--(-.5,-4.25);
\draw (.5,-1)--(.5,-4.25);
\draw (-.5,-4.25)  arc(180:360:.5);
\draw (0,-5) node {$v_1$};
\begin{scope}[shift={(1.75,0)}]
\draw (0.5,-1)  arc(0:180:.5);
\draw (-.5,-1)--(-.5,-4.25);
\draw (.5,-1)--(.5,-4.25);
\draw (-.5,-4.25)  arc(180:360:.5);
\draw (0,-5) node {$v_2$};
\end{scope}
\begin{scope}[shift={(3.5,0)}]
\draw (0.5,-1)  arc(0:180:.5);
\draw (-.5,-1)--(-.5,-4.25);
\draw (.5,-1)--(.5,-4.25);
\draw (-.5,-4.25)  arc(180:360:.5);
\draw (0,-5) node {$v_3$};
\end{scope}
%\draw[black] (A) circle [radius=4pt];
%\draw[black] (AA) circle [radius=4pt];
%\draw[black] (AAA) circle [radius=4pt];
\end{scope}
\end{tikzpicture}
\end{center}
\caption{On the left we have an $S_4$-labeling of $G=P_3$. In the middle we visualize the $S_4$-labeling by replacing each vertex in $V(G)$ with $A=\Z_4$ and drawing the permutation $\pi\in S_4$ on each directed edge $(v_i,v_j)$  by connecting each $a\in \Z_4$ associated to $v_i$ to $\pi(a)$ associated to $v_j$.  This visualization is also known as a cover graph or a lift of $G$ (see~\cite{B17, GR01}). On the right we have a proper $S_4$-$k$-coloring of our $S_4$-labeling for any $k\geq 1$.}
\label{fig:cover_visual}
\end{figure}

We will now define the notion of coloring for $S$-labelings. A \emph{proper $S$-coloring of $(D, \sigma)$} is a mapping $\kappa: V(G) \rightarrow A$ such that for each $(u,v) \in E(D)$ and each coordinate $\pi$ of $\sigma(u,v)$, $\pi(\kappa(u)) \neq \kappa(v)$.
We will call this a \emph{proper $S$-$k$-coloring of $(D, \sigma)$} if $|\kappa(V(G))| \leq k$ which means $\kappa$ uses at most $k$ colors.
 Often we will have $|A|=k$.  We say $G$ is \emph{$S$-$k$-colorable} if there is a proper $S$-$k$-coloring of $(D, \sigma)$ whenever $(D, \sigma)$ is an $S$-labeling of $G$.  
 
Before introducing DP-coloring, it is worth mentioning how some other graph coloring notions can be understood in terms of coloring $S$-labeled graphs.  In the classical sense, a multigraph $G$ is $k$-colorable if and only if $G$ is $I_k$-$k$-colorable. When $S$ is the subset of $S_{[k]}$ consisting of the identity permutation and the permutation of $[k]$ that transposes $2i-1$ and $2i$ for each $i \in [ \lfloor k/2 \rfloor ]$, a simple graph $G$ is \emph{signed $k$-colorable} if and only if $G$ is $S$-$k$-colorable (see~\cite{Z82}).  When $\Gamma$ is an Abelian group of order $k$ and $S = \{ \pi \in S_\Gamma : \text{there is an $a \in \Gamma$ such that $\pi(x) = x + a$ for each $x \in \Gamma$} \}$, a simple graph $G$ is \emph{$\Gamma$-colorable} if and only if $G$ is $S$-$k$-colorable (see~\cite{JL92}). See \cite{JW19} for more examples of older notions of graph colorings expressed as special cases of coloring $S$-labeled graphs.

A multigraph $G$ is \emph{DP-$k$-colorable} if $G$ is $S_A$-$k$-colorable when $|A|=k$.  The \emph{DP-chromatic number} of $G$, denoted $\chi_{DP}(G)$, is the smallest $k$ such that $G$ is DP-$k$-colorable. It is easy to show that $\chi(G) \le \chi_{\ell}(G) \le \chi_{DP}(G)$.

\subsection{Counting S-Colorings and DP Colorings}

Suppose $m \in \mathbb{N}$, $A$ is a finite set, $S$ is some nonempty subset of $S_A$, and $(D, \sigma)$ is an $S$-labeling of the multigraph $G$.  Let $P_S(G, k, (D, \sigma))$ be the number of proper $S$-$k$-colorings of $(D, \sigma)$.  Then, let $P_S(G,k)$ be the minimum of $P_S(G, k, (D, \sigma))$ where the minimum is taken over all $S$-labelings $(D, \sigma)$ of $G$.  Now, assume $|A|=k$.  Notice the chromatic polynomial of $G$ evaluated at $k$, denoted $P(G,k)$, is given by $P(G,k) = P_{I_A}(G,k)$. The \emph{DP color function} of $G$ evaluated at $k$, denoted $P_{DP}(G,k)$, is given by $P_{DP}(G,k) = P_{S_A}(G,k)$. Note that $P_{DP}(G,k) \le P_{S}(G,k)$ for every choice of $S$.

\begin{rmk}\label{rmk:poset}
Given a graph $G$ and $k \in \N$, for every choice of $S$, a subset of $S_A$ where $A$ is a set of $k$ colors, there is a corresponding notion of coloring of $G$. The subset relation on the set of nonempty subsets of $S_A$ induces a partial order on all these notions of coloring. In this poset, DP-$k$-coloring of $G$ is the unique maximal element while classical $k$-coloring of $G$ is a minimal element. A lower bound on $P_S(G,k)$ will also be a lower bound on $P_{S'}(G,k)$ for every $S' \subseteq S$ . In particular, $P_{DP}(G,k)$ is a lower bound on $P_S(G,k)$ for every choice of $S$. This illustrates the importance and difficulty of bounding the DP-color function. As a first step towards understanding the enumerative functions in this poset, in this paper we focus on the DP color function, and on the enumerative function for ``nicer'' colorings, such as field colorings and signed colorings, corresponding to $S$ consisting of linear permutations (see the definition in the next section). 
\end{rmk}

The DP color function was introduced by the second and third named authors in 2019.  It was introduced in hopes of creating a tool for gaining a better understanding of DP-coloring and making progress on some open questions related to the list color function~\cite{KM19}.  Since its introduction in 2019, the DP color function has received some attention in the literature (see e.g.,~\cite{BH21, DKM22, DY21, HK21, KM21, LY22, MT20}).  It is easy to show that for any $k \in \N$, 
$$P_{DP}(G, k) \leq P_\ell(G,k) \leq P(G,k).$$  
Interestingly, unlike the list color function, it is known that $P_{DP}(G,k)$ does not necessarily equal $P(G,k)$ for sufficiently large $k$.  Indeed, in~\cite{KM19} it is shown that if $G$ is a graph with girth that is even, then there is an $N \in \N$ such that $P_{DP}(G, k) < P(G,k)$ whenever $k \geq N$ (this result was further generalized by Dong and Yang in~\cite{DY21}). 

In~\cite{DKM22} we developed a polynomial method for proving lower bounds on $P_{DP}(G,3)$ for certain sparse graphs $G$.  Our main motivation for this paper was to explore generalizations of the method in~\cite{DKM22}, and the notion of $S$-labeled graphs provides the ideal setting for this exploration.  Below we develop a polynomial method for proving some new lower bounds on the number of colorings of certain $S$-labeled graphs. Our results also provide insights into the limitations of this algebraic approach.

\subsection{Outline of the Paper}

We now present an outline of the paper.  In Section~\ref{basic} we introduce the notion of \emph{equivalent $S$-labelings}, and we use this notion to make some fundamental observations that are used throughout the remainder of the paper.  Then, in Section~\ref{DP} we use algebraic techniques to prove two new lower bounds on the DP color function of graphs satisfying certain sparsity conditions when evaluated at a power of a prime.  The key idea is: for a set $A$ with $|A|=k$ and $k$ a power of a prime, we associate with each $S_A$-labeling, $(D, \sigma)$, of a graph $G$ a multivariate polynomial over a finite field with domain $A^{|V(G)|}$ whose non-zeros correspond to proper $S$-$k$-colorings of $(D, \sigma)$.  This idea along with a slightly simplified version, introduced by~\cite{BG22}, of the following well-known result of Alon and F\"{u}redi allows us to establish our bounds.

\begin{thm} [\cite{AF93}] \label{thm: AandF} 
Let $\mathbb{F}$ be an arbitrary field, let $A_1$, $A_2$, $\ldots$, $A_n$ be any non-empty subsets of $\mathbb{F}$, and let $B = \prod_{i=1}^n A_i$.  Suppose that $P \in \mathbb{F}[x_1, \ldots, x_n]$ is a polynomial of degree $d$ that does not vanish on all of $B$.  Then, the number of points in $B$ for which $P$ has a non-zero value is at least $\min \prod_{i=1}^n q_i$ where the minimum is taken over all integers $q_i$ such that $1 \leq q_i \leq |A_i|$ and $\sum_{i=1}^n q_i \geq -d + \sum_{i=1}^n |A_i|$.
\end{thm}

We initially focus on $S_A$-labelings because a lower bound on the DP-color function implies a lower bound on the enumerative functions for all other graph colorings variants that can be expressed in the $S$-labeling context, as well as the list color function (which can not be expressed as a $P_{S}(G,k)$). This includes classical colorings of graphs as well as colorings whose enumerative functions have not even been studied so far. We show the following in Section~\ref{DP}.

\begin{thm} \label{thm:using d=k-2 result} 
Let $k=p^r$ where $p$ is prime, $r \in \N$, and $k > 2$. Suppose $G$ is a connected $n$-vertex simple graph with $m$ edges, and $T$ is a spanning tree of $G$.  Then the following statements hold. 
\\
(i)  Let $q = \lfloor k/2 \rfloor$. Let $G'$ be the multigraph obtained from $G$ by adding $(q - 1)$ parallel edges to each edge $e \in E(G) - E(T)$.  If $\chi_{DP}(G') \leq k$ and $m \leq n(1 + (k-2)/q) -1 + 1/q$, then
$$ P_{DP}(G, k) \geq k^{(n(q+k-2)-qm+1-q)/(k-1)}.$$
\\
(ii)  If $\chi_{DP}(G)\leq k$ and
$m\leq 2n-\frac{k-3}{k-2}$, then
$$
P_{DP}(G, k) \geq k^{((2n-m)(k-2) -(k-3))/(k-1)}.
$$
\end{thm}

It should be noted that Theorem~\ref{thm:using d=k-2 result} remains true if the bound on $m$ is dropped in each statement.  This is because when the bound on $m$ is violated, both parts of Theorem~\ref{thm:using d=k-2 result} yield $P_{DP}(G,k) \geq 1$ which is implied by the bound on the DP-chromatic number.  We however include the bounds on $m$ in the statement of Theorem~\ref{thm:using d=k-2 result} since it can only be useful when these bounds are satisfied.  Notice that these restrictions on $m$ significantly reduce the graphs to which Theorem~\ref{thm:using d=k-2 result} can be applied.  The main obstacle to improving the bound on $m$ is the construction of a polynomial whose non-zeros correspond to proper $S$-colorings.  This polynomial's degree is intimately related to the number of edges in the graph and the number of colors available.  Specifically, when the graph has $m$ edges and there are $k$ colors available in the color set, the polynomial we construct has degree $\lfloor k/2 \rfloor m$ or $(k-2)m$. We further discuss the tightness of the degree of the polynomials we construct at the end of Section~\ref{DP}.

In Section~\ref{sec: linear} we show how we can relax these restrictions on $m$ by a different choice of the $S$-colorings that we are counting.  In particular, suppose $k$ is a power of a prime.  We say a polynomial $f\in \F_k[x,y]$ {\it covers} a permutation $\pi:\F_k\rightarrow\F_k$ if $f(x,\pi(x))=0$ for all $x\in\F_k$. We will call a permutation {\it $\F_k$-linear} if there exists a degree one polynomial in $\F_k[x,y]$ that covers the permutation. Let $\mathcal{L}_k$ be the collection of $\F_k$-linear permutations.  We prove the following theorem with a much improved upper bound on the number of edges.

\begin{thm} \label{thm: linear}
Let $k$ be a power of a prime. If an $n$-vertex graph $G$ with $m$ edges is $S$-$k$-colorable for some $S\subseteq\mathcal{L}_k$ and $m\leq (k-1)n$, then
$$P_S(G,k)\geq k^{n-\frac{m}{k-1}}.$$
In particular, $P_{\mathcal{L}_k}(G,k)\geq k^{n-\frac{m}{k-1}}.$
\end{thm}

Interestingly, we know of no graphs where the bound in Theorem~\ref{thm: linear} does not also hold for the DP color function.  This leads us to make the following provocative conjecture.

\begin{conj} \label{conj: probablyfalse}
If $k$ is a power of a prime,
$$P_{\mathcal{L}_k}(G,k)=P_{DP}(G,k).$$
\end{conj}

It is known that $\mathcal{L}_2 = S_{\F_2}$ and $\mathcal{L}_3 = S_{\F_3}$. 
 So, Conjecture~\ref{conj: probablyfalse} is true for $k=2,3$, but it is open whenever $k$ is a power of a prime satisfying $k \geq 4$. Suppose set $A$ is finite. If true Conjecture~\ref{conj: probablyfalse} would say something rather profound about the $S_A$-labelings of a graph that lead to the fewest colorings (in the case $k$ is a power a prime).  Specifically, it is always possible to find an $S_A$-labeling of the graph that uses only linear permutations and yet has as few colorings as possible. For this reason, it would also be interesting to identify families of graphs where this conjecture is true.

Finally, we end Section~\ref{sec: linear} by discussing an application of Theorem~\ref{thm: linear}. Coloring of signed graphs (\cite{Z82}), and field coloring (\cite{LT21,BG22}) can be viewed as colorings of $S$-labelings when $S \subseteq \mathcal{L}_k$. We focus on colorings of signed graphs and prove a lower bound on the number of colorings of signed graphs. Signed graphs and their colorings have been widely studied since 1980s (see a recent survey in \cite{SV21}), but we do not know of any results on bounding the number of such colorings. We define $P_{\pm}(G,k)$ to the minimum number of $k$-colorings over all signed graphs $(G,\epsilon)$ where signed graphs are defined in Subsection~\ref{subsec:signed}. It follows that for any $k \in \N$, $$P_{DP}(G, k) \leq P_\pm(G,k) \leq P(G,k).$$ 

% \begin{thm}
% \label{thm:signed lower bound}Let $(G,\epsilon)$ be an $n$-vertex signed graph with $m$ edges. 
% Let $k=p^r$ be a power of a prime.  If   $(G,\epsilon)$ has a $k$-coloring and $m\leq (k-1)n$, then
% $$P((G,\epsilon),k)\geq k^{(kn-n-m)/(k-1)}=k^{n-\frac{m}{k-1}}.$$
% \end{thm}

\begin{cor}
\label{cor:signed lower bound}
Let $G$ be an $n$-vertex signed graph with $m$ edges. 
Let $k=p^r$ be a power of a prime.  If  $\chi_s(G)\leq k$ and $m\leq (k-1)n$, then 
$$P_{\pm}(G,k)\geq k^{n-\frac{m}{k-1}}.$$
\end{cor}

In Section~\ref{planar} we apply our results to families of sparse planar graphs.  When Thomassen~\cite{T94} proved all planar graphs are 5-choosable in 1994, it was asked whether there are exponentially many 5-list-colorings of planar graphs. In 2007, Thomassen~\cite{T07b} proved that this is true by showing $P_{\ell}(G,5) \ge 2^{n/9}$ when $G$ is an $n$-vertex planar graph, and correspondingly for DP-colorings, Postle and Smith-Roberge~\cite{PS23} recently showed that $P_{DP}(G,5) \ge 2^{n/67}$.  

When sparsity is also controlled by forbidding short cycles, Thomassen showed that for $n$-vertex planar graphs $G$ of girth 5, $P_{\ell}(G,3) \ge 2^{n/1000}$~\cite{T07} for any such graph $G$. This was recently improved by Postle and Smith-Roberge~\cite{PS23} to $P_{DP}(G,3) \ge 2^{n/282}$. All of these results use intricate and long topological arguments. In a previous paper~\cite{DKM22}, we improved this further to $P_{DP}(G,3) \ge 3^{n/6}$. Now, using the technique outlined in Theorem~\ref{thm:using d=k-2 result}, we can show there are exponentially many such DP-colorings for many such families of sparse planar graphs.  

In Section~\ref{planar} we show how Theorem~\ref{thm:using d=k-2 result} can be used to prove new lower bounds on the DP color function of families of planar graphs that forbid certain short cycles. Then, we use Corollary~\ref{cor:signed lower bound} to prove analogous results in the setting of signed coloring. These families of planar graphs have previously been studied in the literature but only from the perspective of their various chromatic numbers. There are no previously known lower bounds, much less exponential lower bounds, on the enumerative functions of any notion of coloring for these classes of planar graphs that are considered in our applications. 

We end Section~\ref{planar} by considering triangle-free planar graphs.  Gr\"{o}tzsch's famous theorem~\cite{G59} states that triangle-free planar graphs are 3-colorable. However, there exist triangle-free planar graphs are that are not 3-choosable and hence not DP-3-colorable. By a degeneracy argument, it follows that such graphs are 4-choosable and DP-4-colorable. So, it's natural to ask whether there are exponentially many list-4-colorings and exponentially many DP-4-colorings of such graphs. 

\begin{conj}\label{conj: triangle-free planar}
    There exists a constant $c >1$, such that for any triangle-free planar $n$-vertex graph $G$, $P_{DP}(G,4) \ge c^n$. 
\end{conj}

This conjecture is in contrast to the recent results that show that there are at most subexponentially many number of 3-colorings of these graphs (see~\cite{ADPT13, T23, DP22}). Towards this conjecture, we show the following.

\begin{thm}\label{thm: Planar4}
Let $G$ be an $n$-vertex triangle-free planar graph with $m$ edges. The following statements are true.
\begin{enumerate}[(i)]
\item $P_{\ell}(G,4)\geq 4^{\frac{n+4}{3}}$.

\item $P_{\pm}(G,4)\geq 4^{\frac{n+4}{3}}$.

\item Suppose there exists $c>0$ such that $m \le (2-c)n$ and $cn \geq 1/2$, then $P_{DP}(G,4) \ge 4^{(4cn-1)/3}.$

\end{enumerate}

\end{thm}

%%%%%%%%%%%%%%%%%%%%%%%%%
%%%%%%%%%%%%%%%%%%%%%%%%%
\section{Equivalent S-labelings}\label{basic}
Throughout this section we will assume $G$ is a simple graph. Suppose $A$ is a finite set, $V(G) = \{v_1, \ldots, v_n \}$, and  $S$ and $S'$ are subsets of $S_A$.  We say that an $S$-labeling $(D,\sigma)$ of $G$ is {\it equivalent to} an $S'$-labeling $(D',\sigma')$ of $G$ if the following holds.  There is $(\tau_1, \ldots, \tau_n) \in (S_A)^n$ and for each pair of adjacent vertices $v_i$ and $v_j$ in $G$ the following holds.  If $(v_i,v_j) \in E(D)$ and $(v_i,v_j) \in E(D')$, then  $\sigma(v_i,v_j)(s) = t$ if and only if $\sigma'(v_i,v_j)(\tau_i(s))=\tau_j(t)$ for all $s,t\in A$.  If $(v_i,v_j) \in E(D)$ and $(v_j,v_i) \in E(D')$, then  $\sigma(v_i,v_j)(s) = t$ if and only if $(\sigma'(v_j,v_i))^{-1}(\tau_i(s))=\tau_j(t)$ for all $s,t\in A$. See Figure~\ref{fig:equiv_cover}.

\begin{obs} \label{obs: construct}
Suppose $A$ is a finite set and $G$ is a simple graph with vertex set $V(G) = \{v_1, \ldots, v_n \}$.  Suppose $(\tau_1, \ldots, \tau_n) \in S_A^n$. Any $S_A$-labeling $(D, \sigma)$ is equivalent to the $S_A$-labeling $(D, \sigma')$ where $\sigma' =\tau_j\circ\sigma(v_i,v_j)\circ\tau_i^{-1}$ whenever $(v_i,v_j) \in E(D)$.
\end{obs}

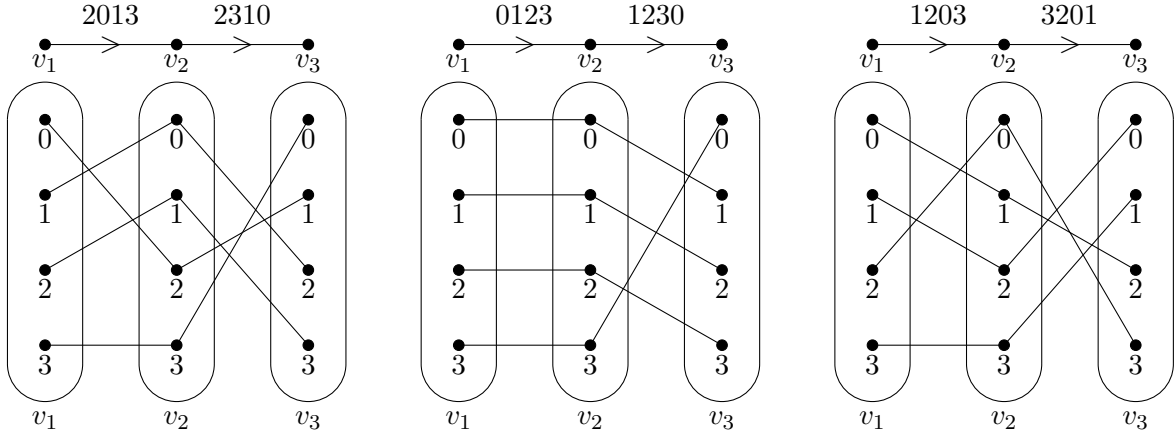
\begin{figure}[t]
\begin{center}
\begin{tikzpicture}
\coordinate (A) at (0,0);
\coordinate (B) at (1.75,0);
\coordinate (C) at (3.5,0);
\draw (A)--(B)--(C);
\filldraw[black] (A) circle [radius=2pt] node[below] {$v_1$};
\filldraw[black] (B) circle [radius=2pt] node[below] {$v_2$};
\filldraw[black] (C) circle [radius=2pt] node[below] {$v_3$};
\draw (.875,0) node {$>$};
\draw (.875,.4) node {$2013$};
\draw (2.625,0) node {$>$};
\draw (2.625,.4) node {$2310$};
\coordinate (A) at (0,-1);
\coordinate (B) at (0,-2);
\coordinate (C) at (0,-3);
\coordinate (D) at (0,-4);
\coordinate (AA) at (1.75,-1);
\coordinate (BB) at (1.75,-2);
\coordinate (CC) at (1.75,-3);
\coordinate (DD) at (1.75,-4);
\coordinate (AAA) at (3.5,-1);
\coordinate (BBB) at (3.5,-2);
\coordinate (CCC) at (3.5,-3);
\coordinate (DDD) at (3.5,-4);
\draw (A)--(CC)--(BBB);
\draw (B)--(AA)--(CCC);
\draw (C)--(BB)--(DDD);
\draw (D)--(DD)--(AAA);
\filldraw[black] (A) circle [radius=2pt] node[below] {$0$};
\filldraw[black] (B) circle [radius=2pt] node[below] {$1$};
\filldraw[black] (C) circle [radius=2pt] node[below] {$2$};
\filldraw[black] (D) circle [radius=2pt] node[below] {$3$};
\filldraw[black] (AA) circle [radius=2pt] node[below] {$0$};
\filldraw[black] (BB) circle [radius=2pt] node[below] {$1$};
\filldraw[black] (CC) circle [radius=2pt] node[below] {$2$};
\filldraw[black] (DD) circle [radius=2pt] node[below] {$3$};
\filldraw[black] (AAA) circle [radius=2pt] node[below] {$0$};
\filldraw[black] (BBB) circle [radius=2pt] node[below] {$1$};
\filldraw[black] (CCC) circle [radius=2pt] node[below] {$2$};
\filldraw[black] (DDD) circle [radius=2pt] node[below] {$3$};
\begin{scope}[shift={(0,0)}]
\draw (0.5,-1)  arc(0:180:.5);
\draw (-.5,-1)--(-.5,-4.25);
\draw (.5,-1)--(.5,-4.25);
\draw (-.5,-4.25)  arc(180:360:.5);
\draw (0,-5) node {$v_1$};
\end{scope}
\begin{scope}[shift={(1.75,0)}]
\draw (0.5,-1)  arc(0:180:.5);
\draw (-.5,-1)--(-.5,-4.25);
\draw (.5,-1)--(.5,-4.25);
\draw (-.5,-4.25)  arc(180:360:.5);
\draw (0,-5) node {$v_2$};
\end{scope}
\begin{scope}[shift={(3.5,0)}]
\draw (0.5,-1)  arc(0:180:.5);
\draw (-.5,-1)--(-.5,-4.25);
\draw (.5,-1)--(.5,-4.25);
\draw (-.5,-4.25)  arc(180:360:.5);
\draw (0,-5) node {$v_3$};
\end{scope}

%%%%%%%%%%%%%%%%%%%%%%%%%%
\begin{scope}[shift={(5.5,0)}]
\coordinate (A) at (0,0);
\coordinate (B) at (1.75,0);
\coordinate (C) at (3.5,0);
\draw (A)--(B)--(C);
\filldraw[black] (A) circle [radius=2pt] node[below] {$v_1$};
\filldraw[black] (B) circle [radius=2pt] node[below] {$v_2$};
\filldraw[black] (C) circle [radius=2pt] node[below] {$v_3$};
\draw (.875,0) node {$>$};
\draw (.875,.4) node {$0123$};
\draw (2.625,0) node {$>$};
\draw (2.625,.4) node {$1230$};
\coordinate (A) at (0,-1);
\coordinate (B) at (0,-2);
\coordinate (C) at (0,-3);
\coordinate (D) at (0,-4);
\coordinate (AA) at (1.75,-1);
\coordinate (BB) at (1.75,-2);
\coordinate (CC) at (1.75,-3);
\coordinate (DD) at (1.75,-4);
\coordinate (AAA) at (3.5,-1);
\coordinate (BBB) at (3.5,-2);
\coordinate (CCC) at (3.5,-3);
\coordinate (DDD) at (3.5,-4);
\draw (A)--(AA)--(BBB);
\draw (B)--(BB)--(CCC);
\draw (C)--(CC)--(DDD);
\draw (D)--(DD)--(AAA);
\filldraw[black] (A) circle [radius=2pt] node[below] {$0$};
\filldraw[black] (B) circle [radius=2pt] node[below] {$1$};
\filldraw[black] (C) circle [radius=2pt] node[below] {$2$};
\filldraw[black] (D) circle [radius=2pt] node[below] {$3$};
\filldraw[black] (AA) circle [radius=2pt] node[below] {$0$};
\filldraw[black] (BB) circle [radius=2pt] node[below] {$1$};
\filldraw[black] (CC) circle [radius=2pt] node[below] {$2$};
\filldraw[black] (DD) circle [radius=2pt] node[below] {$3$};
\filldraw[black] (AAA) circle [radius=2pt] node[below] {$0$};
\filldraw[black] (BBB) circle [radius=2pt] node[below] {$1$};
\filldraw[black] (CCC) circle [radius=2pt] node[below] {$2$};
\filldraw[black] (DDD) circle [radius=2pt] node[below] {$3$};
\begin{scope}[shift={(0,0)}]
\draw (0.5,-1)  arc(0:180:.5);
\draw (-.5,-1)--(-.5,-4.25);
\draw (.5,-1)--(.5,-4.25);
\draw (-.5,-4.25)  arc(180:360:.5);
\draw (0,-5) node {$v_1$};
\end{scope}
\begin{scope}[shift={(1.75,0)}]
\draw (0.5,-1)  arc(0:180:.5);
\draw (-.5,-1)--(-.5,-4.25);
\draw (.5,-1)--(.5,-4.25);
\draw (-.5,-4.25)  arc(180:360:.5);
\draw (0,-5) node {$v_2$};
\end{scope}
\begin{scope}[shift={(3.5,0)}]
\draw (0.5,-1)  arc(0:180:.5);
\draw (-.5,-1)--(-.5,-4.25);
\draw (.5,-1)--(.5,-4.25);
\draw (-.5,-4.25)  arc(180:360:.5);
\draw (0,-5) node {$v_3$};
\end{scope}
\end{scope}
%%%%%%%%%%%%%%%%%%%%%%%%%%
\begin{scope}[shift={(11,0)}]
\coordinate (A) at (0,0);
\coordinate (B) at (1.75,0);
\coordinate (C) at (3.5,0);
\draw (A)--(B)--(C);
\filldraw[black] (A) circle [radius=2pt] node[below] {$v_1$};
\filldraw[black] (B) circle [radius=2pt] node[below] {$v_2$};
\filldraw[black] (C) circle [radius=2pt] node[below] {$v_3$};
\draw (.875,0) node {$>$};
\draw (.875,.4) node {$1203$};
\draw (2.625,0) node {$>$};
\draw (2.625,.4) node {$3201$};
\coordinate (A) at (0,-1);
\coordinate (B) at (0,-2);
\coordinate (C) at (0,-3);
\coordinate (D) at (0,-4);
\coordinate (AA) at (1.75,-1);
\coordinate (BB) at (1.75,-2);
\coordinate (CC) at (1.75,-3);
\coordinate (DD) at (1.75,-4);
\coordinate (AAA) at (3.5,-1);
\coordinate (BBB) at (3.5,-2);
\coordinate (CCC) at (3.5,-3);
\coordinate (DDD) at (3.5,-4);
\draw (A)--(BB)--(CCC);
\draw (B)--(CC)--(AAA);
\draw (C)--(AA)--(DDD);
\draw (D)--(DD)--(BBB);
\filldraw[black] (A) circle [radius=2pt] node[below] {$0$};
\filldraw[black] (B) circle [radius=2pt] node[below] {$1$};
\filldraw[black] (C) circle [radius=2pt] node[below] {$2$};
\filldraw[black] (D) circle [radius=2pt] node[below] {$3$};
\filldraw[black] (AA) circle [radius=2pt] node[below] {$0$};
\filldraw[black] (BB) circle [radius=2pt] node[below] {$1$};
\filldraw[black] (CC) circle [radius=2pt] node[below] {$2$};
\filldraw[black] (DD) circle [radius=2pt] node[below] {$3$};
\filldraw[black] (AAA) circle [radius=2pt] node[below] {$0$};
\filldraw[black] (BBB) circle [radius=2pt] node[below] {$1$};
\filldraw[black] (CCC) circle [radius=2pt] node[below] {$2$};
\filldraw[black] (DDD) circle [radius=2pt] node[below] {$3$};
\begin{scope}[shift={(0,0)}]
\draw (0.5,-1)  arc(0:180:.5);
\draw (-.5,-1)--(-.5,-4.25);
\draw (.5,-1)--(.5,-4.25);
\draw (-.5,-4.25)  arc(180:360:.5);
\draw (0,-5) node {$v_1$};
\end{scope}
\begin{scope}[shift={(1.75,0)}]
\draw (0.5,-1)  arc(0:180:.5);
\draw (-.5,-1)--(-.5,-4.25);
\draw (.5,-1)--(.5,-4.25);
\draw (-.5,-4.25)  arc(180:360:.5);
\draw (0,-5) node {$v_2$};
\end{scope}
\begin{scope}[shift={(3.5,0)}]
\draw (0.5,-1)  arc(0:180:.5);
\draw (-.5,-1)--(-.5,-4.25);
\draw (.5,-1)--(.5,-4.25);
\draw (-.5,-4.25)  arc(180:360:.5);
\draw (0,-5) node {$v_3$};
\end{scope}
\end{scope}
\end{tikzpicture}
\end{center}
\caption{In the top row we have three equivalent  $S_4$-labelings of $G=P_3$ where $A=\Z_4$ and below we have their visualization as described in Figure~\ref{fig:cover_visual}. On the left we start with $(D,\sigma)$. In the middle we have the equivalent $S_4$-labeling from $(\tau_1,\tau_2,\tau_3)=(0123,1203,0123)$. This is $(D,\sigma_{v_2})$, which is an example of Lemma~\ref{lem:equal_covers}(i) with chosen vertex $v_2$ and permutation $\alpha=1203$. On the right we have the equivalent $S_4$-labeling from $(\tau_1,\tau_2,\tau_3)=(1023,1023,1023)$. This is $(D,\sigma_{\alpha})$, which is an example of Lemma~\ref{lem:equal_covers}(ii)  where  $\alpha=1023$.}
\label{fig:equiv_cover}
\end{figure}

For a fixed orientation $D$ of the graph $G$, we say $G$ is \emph{$(D,S)$-$k$-colorable}  if there is a proper $S$-$k$-coloring of $(D, \sigma)$ whenever $(D, \sigma)$ is an $S$-labeling of $G$.  Notice that when $S$ is a nonempty subset of $S_A$ with the property that if $\pi \in S$, then $\pi^{-1} \in S$, our chosen orientation doesn't matter since we can form an equivalent $S$-labeling by changing the orientation of the edge and the associated permutation $\pi$ with $\pi^{-1}$.  We now formally state this idea.

\begin{obs} \label{obs: orient}
Suppose $S$ is a nonempty subset of $S_A$ with the property that if $\pi \in S$, then $\pi^{-1} \in S$.  Suppose $D$ is a fixed orientation of $G$.  If $G$ is $(D,S)$-$k$-colorable, then $G$ is $S$-$k$-colorable. Moreover, if $|A| \geq m$, there is an $S$-labeling of $G$, $(D, \sigma)$, satisfying $P_{S}(G, m) = P_{S}(G, m, (D,\sigma))$.
\end{obs}

 When $S$ satisfies the hypotheses of Observation~\ref{obs: orient}, we will always let $D$ be the orientation of $G$ obtained by orienting each edge $v_iv_j \in E(G')$ with $j > i$ so that $v_i$ is the tail and $v_j$ is the head.  We will refer to $D$ as the \emph{canonical orientation} of $G$.  Then, in order to show that $G$ is $S$-$k$-colorable, we will show that $G$ is $(D,S)$-$k$-colorable.

\begin{lem}\label{lem:equivalent S-labeling}
Let $A$ be a finite set. Given an $S$-labeling $(D,\sigma)$ of a graph $G$ where $S\subseteq S_A$, then  we can construct equivalent $S_A$-labelings:
\begin{enumerate}[(i)]
    \item $(D,\sigma_u)$ for a chosen vertex $u\in V(G)$ and permutation $\alpha\in S_A$. We set $\sigma_u(e)$ equal to $\sigma(e)$ except for edges $e$ that are incident to $u$. If $e = (u,v)$, let $\sigma_u(e)=\sigma(e)\circ \alpha^{-1}$. If $e = (v,u)$ let $\sigma_u(e)=\alpha\circ \sigma(e)$. %We certainly have that $(D,\sigma_u)$ is a $S_k$-labeling, but may not still be an $S$-labeling. 
    \item $(D,\sigma_{\alpha})$ for a chosen permutation $\alpha\in S_A$. For any edge $e$ we set $\sigma_{\alpha}(e)=\alpha^{-1}\circ \sigma(e)\circ \alpha$. Note that  $(D,\sigma_{\alpha})$ is an $S_A$-labeling, an $\alpha^{-1}S\alpha$-labeling, and will also be an $S$-labeling if $\alpha\in S$ and $S$ is closed under conjugation~\footnote{This means that whenever $\beta, \gamma \in S$, $(\beta^{-1} \circ \gamma \circ \beta) \in S$.}. 
\end{enumerate}
\label{lem:equal_covers}
\end{lem}
\begin{proof}
Suppose we have an $S$-labeling $(D,\sigma)$ of a graph $G$ where $S\subseteq S_A$. 
For part (i), for any vertex $v_i$ and a choice of $\alpha\in S_A$ we choose the tuple $(\tau_1, \ldots, \tau_n) \in S_A^n$ such that $\tau_i=\alpha$ and $\tau_j$ is the identity for all $i\neq j$. It follows that $(D,\sigma)$ is equivalent to $(D,\sigma_{v_i})$ since $\sigma_{v_i}$ equals the $\sigma'$ specified by Observation~\ref{obs: construct} with respect to our chosen $(\tau_1, \ldots, \tau_n)$.
For part (ii), for any $\alpha\in S_A$ we choose the tuple $(\tau_1, \ldots, \tau_n) \in S_A^n$ such that  $\tau_j=\alpha$ for all $j$. It follows that $(D,\sigma)$ is equivalent to $(D,\sigma_{\alpha})$ since $\sigma_{\alpha}$ equals the $\sigma'$ specified by Observation~\ref{obs: construct} with respect to our chosen $(\tau_1, \ldots, \tau_n)$.
\end{proof}

\begin{cor}
    Suppose $A$ is a finite set.  For any $S_A$-labeling  $(D,\sigma)$ of $G$ and spanning forest $F$ of $G$ there is an equivalent $S_A$-labeling $(D,\sigma')$ where $\sigma'(e)$ is the identity permutation for all edges $e \in E(F)$. 
\label{cor:id_on_tree}
\end{cor}
\begin{proof}
    We can use Lemma~\ref{lem:equal_covers} part (i) inductively on the number of vertices in $G$.
\end{proof}

Several researchers have similar versions of Corollary~\ref{cor:id_on_tree} in different contexts~\cite{GR01,DP15,KM19,KM21}. 

\section{Counting DP-colorings} \label{DP}

In this section we will assume that $G$ is a simple graph unless otherwise noted, $V(G)=\{v_1,\ldots, v_n\}$, and $|E(G)|=m$.

Suppose $\F$ is an arbitrary field, $A \subseteq \F$, and $|A|=k$.  Suppose $(D,\sigma)$ is an $S_A$-labeling of $G$.  We say that $f \in \F[x_1, \ldots, x_n]$ \emph{covers} $(D, \sigma)$ if $f(a_1, \ldots, a_n) \neq 0$ implies that $\kappa:V(G)\rightarrow A$ given by $\kappa(v_i)=a_i$, is a proper $S_A$-$k$-coloring of $G$.  Now, let $\mathcal{F}_{G,(D,\sigma)}$ be the set of all polynomials in $\F[x_1, \ldots, x_n]$ that cover $(D, \sigma)$ and are nonzero for at least one element in $A^n$.   

As mentioned above, we use a well-known result of Alon and F\"{u}redi, Theorem~\ref{thm: AandF}, to establish a non-trivial lower bounds on the DP color function of certain graphs.  For our purposes, it will be easier to apply the following weaker version of Theorem~\ref{thm: AandF}.

\begin{cor} [B. Bosek, J. Grytczuk, G. Gutowski, O. Serra, M. Zajac~\cite{BG22}] \label{thm: bound}
Let $\mathbb{F}$ be an arbitrary field, let $A_1$, $A_2$, $\ldots$, $A_n$ be any non-empty subsets of $\mathbb{F}$, and let $B = \prod_{i=1}^n A_i$.  Suppose that $P \in \mathbb{F}[x_1, \ldots, x_n]$ is a polynomial of degree $d$ that does not vanish on all of $B$.  If $S = \sum_{i=1}^n |A_i|$, $t = \max |A_i|$, $S \geq n + d$, and $t \geq 2$, then the number of points in $B$ for which $P$ has a non-zero value is at least $t^{(S-n-d)/(t-1)}.$
\end{cor}

We begin by illustrating the ease with which this result applies to the list color function because we can work with the graph polynomial over the reals. A version of this proposition for planar graphs was also observed in~\cite{BG22}, the paper that initiated the idea of using Theorem~\ref{thm: AandF} to bound the list color function of planar graphs.

\begin{pro}\label{prop: listcoloring}
Suppose $G$ is an $n$-vertex graph with $m$ edges, and $k$ is a positive integer greater than 1 satisfying $\chi_{\ell}(G) \leq k$.  If $m \leq (k-1)n$, then 
$$P_{\ell}(G,k) \geq k^{n-\frac{m}{k-1}}.$$
\end{pro}

\begin{proof}
Suppose $V(G) = \{v_1, \ldots, v_n \}$.  Furthermore suppose that $L$ is a $k$-assignment for $G$ such that $L(v_i) \subset \R$ for each $i \in [n]$ and $P(G,L) = P_{\ell}(G,k)$.  Now, suppose $f \in \R[x_1, \ldots, x_n]$ is given by
$$f(x_1, \ldots, x_n) = \prod_{v_iv_j \in E(G), i < j} (x_i-x_j).$$
Clearly, $f$ is of degree $m$.  For each $i \in [n]$ let $A_i = L(v_i)$, and let $B = \prod_{i=1}^n A_i$.  By the formula for $f$, we have that for any $(b_1, \ldots, b_n) \in B$, $f(b_1, \ldots, b_n) \neq 0$ if and only if the function $g: V(G) \rightarrow \R$ given by $g(v_i)=b_i$ for each $i \in [n]$ is a proper $L$-coloring of $G$.

Consequently, $P(G,L)$ is the number of elements in $B$ for which $f$ has a nonzero value.  Since $\chi_{\ell}(G) \leq k$, we know that $f$ does not vanish on all of $B$.  Finally, Corollary~\ref{thm: bound} yields the desired result. 
\end{proof}

In Section~\ref{sec: linear}, with substantially more work, we will show the same bounds (on the number of edges and on the number of colorings) work when we consider colorings of $S$-labeled graphs for $S$ consisting of linear permutations.  

In order to apply Corollary~\ref{thm: bound} to bound the DP-color function (or, more generally, the number of colorings of $S$-labeled graphs), it is not enough to work with the graph polynomial over the reals. The design of the applicable polynomial over a finite field is more involved as we describe in the rest of this section. 

We take $\F=\F_{k}$ where $k$ is a power of a prime, $A_i=\F_k$ and  $B = \F_k^n$.  Then, we define for any $S_{\F}$-labeling $(D, \sigma)$ of a graph $G$  a polynomial $f_{(D,\sigma)} \in \mathcal{F}_{G,(D,\sigma)}$ of some degree $d>0$. This makes $S=kn$ and $t=k$.  Finally, in the case that there is a proper $S_{\F}$-coloring of $(D, \sigma)$, Corollary~\ref{thm: bound} gives a lower bound of $k^{(kn-n-d)/(k-1)}$ proper $S_{\F}$-colorings of $(D, \sigma)$ provided $d \leq (k-1)n$.  

In the next few propositions we will be defining polynomials over $\F_k[x,y]$ that are zero on  points $(c,\pi(c))$ for a permutation $\pi\in S_{\F_k}$ for all $c\in \F_k$. Our building blocks will be {\it $L$-polynomials}. 
An $L$-polynomial is a polynomial in $\F_k[x,y]$ constructed from $i,j \in \F_k$ and $\pi\in S_{\F_k}$ given by
$$L^{\pi}_{i,j}(x,y):=(j-i)(y-\pi(i))-(\pi(j)-\pi(i))(x-i).$$
Essentially, $L^\pi_{i,j}(x,y)$ will be zero on all points that lie on the line between the points $(i,\pi(i))$ and $(j,\pi(j))$.  Also, when $i \neq j$, we define the \emph{permutation of $\F_k$ corresponding to $L^\pi_{i,j}(x,y)$} as the function $\sigma^\pi_{i,j} : \F_k \rightarrow \F_k$ given by
$$\sigma^\pi_{i,j}(x) = (j-i)^{-1}(\pi(j)-\pi(i))(x-i) + \pi(i).$$
Notice that $\sigma_{i,j}^\pi$ is a permutation of $\F_k$ and $L^\pi_{i,j}(x,\sigma^\pi_{i,j}(x))=0$ for each $x \in \F_k$. We will discuss these permutations more in Section~\ref{sec: linear}.

\begin{pro} \label{prop: edgecoloring}
Let $\pi:\F_k\rightarrow \F_k$ be a permutation with $k\geq 2$. There exists a polynomial $f_{\pi} \in \mathbb{F}_k[x,y]$ of degree $d=\lfloor k/2 \rfloor$ that is a product of $L$-polynomials such that $f_{\pi}(c,\pi(c))=0$ for all $c\in\F_k$. 
\end{pro}

\begin{proof}
Let $\pi:\F_k\rightarrow \F_k$ be a permutation with $k\geq 2$. Recall that that $L^\pi_{i,j}(i,\pi(i))=0$ and $L^\pi_{i,j}(j,\pi(j))=0$. 

When $k$ is even we arbitrarily pair off the elements in $\F_k$ so that $\F_k=\{s_1,t_1,s_2,t_2,\ldots, s_{k/2},t_{k/2}\}$.
Then, let 
$$f_{\pi}(x,y)=\prod_{i=1}^{k/2}L^\pi_{s_i,t_i}(x,y).$$
This has degree $\lfloor k/2 \rfloor$ and $f_{\pi}(c,\pi(c))=0$ for all $c\in\F_k$.

Suppose $k$ is odd. We will consider a copy of the complete graph $K_k$ with vertices $\F_k$. Label the edge between $a\in \F_k$ and $b\in \F_k$  with $(\pi(b)-\pi(a))/(b-a)$. Note that this label is never 0, so all edges have labels from $\F_k-\{0\}$. This means we are using at most $(k-1)$ edge labels. Since $K_k$ is not $(k-1)$-edge colorable (\cite{BCC67}) we know that there must be some vertex $a$ and two edges $ab_1$ and $ab_2$ with the same edge label. Thus, $L^\pi_{a,b_1}(b_2,\pi(b_2))=0$. Since $k$ is odd, we can pair off elements of $\F_k-\{a,b_1,b_2\}$ so that $\F_k-\{a,b_1,b_2\}=\{s_1,t_1,\ldots, s_{(k-3)/2},t_{(k-3)/2}\}$. Then, let 
$$f_{\pi}=L^\pi_{a,b_1}(x,y)\prod_{i=1}^{(k-3)/2}L^\pi_{s_i,t_i}(x,y).$$
This has degree $\lfloor k/2 \rfloor$ and $f_{\pi}(c,\pi(c))=0$ for all $c\in\F_k$.
\end{proof}

We have verified via a computer search that $\lfloor k/2 \rfloor$ is the best degree achievable above when $k \le 7$.  However, the achievable degree increases to $k-2$ when we require the polynomial to have at least one specified nonzero, as shown in Proposition~\ref{prop:f_pi_for_simple} below.  Recall that our big picture goal is to construct polynomials in $\mathcal{F}_{G,(D,\sigma)}$.  The issue with directly using Proposition~\ref{prop: edgecoloring}, to construct a polynomial in $\mathcal{F}_{G,(D,\sigma)}$ is that we can't directly guarantee the existence of a nonzero on a proper $S_A$-coloring of $(D, \sigma)$ without an additional hypothesis.  We will present two different possible additional hypotheses to overcome this difficulty.  The second hypothesis we describe involves specifying a nonzero (Proposition~\ref{prop:f_pi_for_simple} below).  The first allows us to directly work with the polynomials in Proposition~\ref{prop: edgecoloring} by working with multigraphs.  Our next observation and lemma describes the details.

\begin{obs}
Let $k$ be a power of a prime, $A=\F_k$, $G$ be a connected $n$-vertex simple graph with $m$ edges, $T$ be a spanning tree of $G$, and  $D$ be the canonical orientation of $G$. Suppose that $(D, \sigma)$ is an $S_{A}$-labeling of $G$ with the property that $\sigma(e)$ is the identity permutation whenever $e \in E(T)$.  Let $G'$ be the multigraph obtained from $G$ by adding $(\lfloor k/2 \rfloor - 1)$ parallel edges to each edge $e \in E(G) - E(T)$.

Let $q = \lfloor k/2 \rfloor$.  Whenever $v_iv_j \in E(T)$ with $j > i$, let $f_{\sigma(v_iv_j)}=x_i-x_j$.  Now, suppose $v_iv_j \in E(G)-E(T)$ with $j>i$ and $\pi = \sigma(v_iv_j)$. We let $f_{\sigma(v_iv_j)}$ be the product of the $q$ $L$-polynomials in the tuple $(L^\pi_{s_1,t_1}(x_i,x_j), \ldots, L^\pi_{s_q,t_q}(x_i,x_j))$ where $s_r \neq t_r$ for each $r \in [q]$ and $f_{\sigma(v_iv_j)}(a,\sigma(v_iv_j)(a))=0$ for each $a\in \F_k$.  The existence of such a tuple is guaranteed by the proof of Proposition~\ref{prop: edgecoloring}.

Now, let $\sigma'$ be the function with domain $E(D)$ given by $\sigma'(v_iv_j) = (\sigma(v_iv_j))$ whenever $v_iv_j \in E(T)$ and $\sigma'(v_iv_j) = (\sigma^\pi_{s_1,t_1}(x_i,x_j), \ldots, \sigma^\pi_{s_q,t_q}(x_i,x_j))$ whenever $v_iv_j \in E(G)-E(T)$.

Now, $(D, \sigma')$ is an $S_A$-labeling of $G'$.  If we let $$f_{(D,\sigma)}=\prod_{(v_i,v_j)\in E(D)}f_{\sigma(v_iv_j)},$$
then $f_{(D, \sigma)}(g(v_1), \ldots, g(v_n))\neq 0$ if and only if $g: V(G') \rightarrow A$ is a proper $S_A$-coloring of $(D, \sigma')$.  
\end{obs}

\begin{lem} \label{lem: dpchromaticnumber}
Let $k$ be a power of a prime, $A=\F_k$, $G$ be a connected $n$-vertex simple graph with $m$ edges, $T$ be a spanning tree of $G$, and  $D$ be the canonical orientation of $G$. Suppose that $(D, \sigma)$ is an $S_{A}$-labeling of $G$ with the property that $\sigma(e)$ is the identity permutation whenever $e \in E(T)$.  Let $G'$ be the multigraph obtained from $G$ by adding $(\lfloor k/2 \rfloor - 1)$ parallel edges to each edge $e \in E(G) - E(T)$.  If $\chi_{DP}(G') \leq k$ and $(D, \sigma)$ is an $S_{A}$-labeling of $G$, then there exists a polynomial $f \in \mathcal{F}_{G,(D,\sigma)}$  of degree  $
\lfloor k/2 \rfloor (m-n+1)+n-1$. 
\end{lem}

\begin{proof}

Let $V(G)=\{v_1,v_2,\ldots v_n\}$. We will now define a polynomial for each edge. Whenever $v_iv_j \in E(T)$ with $j > i$, let $f_{\sigma(v_iv_j)}=x_i-x_j$.  

For all other edges $v_iv_j \in E(G)-E(T)$ with $j>i$,  Proposition~\ref{prop: edgecoloring} gives us a function $f_{\sigma(v_iv_j)} \in \F_k[x_i,x_j]$ that has $f_{\sigma(v_iv_j)}(a,\sigma(v_iv_j)(a))=0$ for any $a\in \F_k$. We are also guaranteed $f_{\sigma(v_iv_j)}$ has degree $\lfloor k/2 \rfloor$. Let $$f_{(D,\sigma)}=\prod_{(v_i,v_j)\in E(D)}f_{\sigma(v_iv_j)}.$$
Notice that there is an $S_A$-labeling $(D, \sigma')$ of $G'$ such that $f_{(D, \sigma)}(g(v_1), \ldots, g(v_n))\neq 0$ if and only if $g: V(G') \rightarrow \F_k$ is a proper $S_A$-coloring of $(D, \sigma')$.  Moreover, if $g: V(G') \rightarrow \F_k$ is a proper $S_A$-coloring of $(D, \sigma')$, then $g$ is also a proper $S_A$-coloring of $(D, \sigma)$. 
\end{proof}

Now, we present the counterparts of Proposition~\ref{prop: edgecoloring} and Lemma~\ref{lem: dpchromaticnumber}, where we avoid working with multigraphs by allowing a higher degree for our polynomials.

\begin{pro} \label{prop:f_pi_for_simple}
Let $\pi:\F_k\rightarrow \F_k$ be a permutation, $k\geq 3$ and $(a,b)\in \F_k^2$ be an ordered pair such that $\pi(a)\neq b$. There exists a polynomial $f_{\pi} \in \mathbb{F}_k[x,y]$ of degree $d=k-2$ that is a product of $L$-polynomials such that $f_{\pi}(c,\pi(c))=0$ for all $c\in\F_k$ and $f_{\pi}(a,b)\neq 0$.
\end{pro}

\begin{proof} Let $\pi:\F_k\rightarrow \F_k$  be a permutation and $(a,b)\in \F_k^2$ be an ordered pair such that $\pi(a)\neq b$. Observation~14 in~\cite{KM20} implies the desired result when $k=3$. So, suppose $k>3$.
Notice that $L^\pi_{a,j}(x,y)$ is not zero at $(a,b)$ if $j\neq a$ and $L^\pi_{\pi^{-1}(b),j}(x,y)$ is not zero at $(a,b)$ if $j\neq \pi^{-1}(b)$. 
Now, choose two distinct elements $s$ and $t$ from $\F_k-\{a,\pi^{-1}(b)\}$. 
Let
$$f_{\pi}(x,y)=L^\pi_{a,s}(x,y)L^\pi_{\pi^{-1}(b),t}(x,y)\prod_{j\neq s,t,a,\pi^{-1}(b)}L^\pi_{a,j}(x,y).$$
By our previous notes, this has all the required properties and has degree $k-2$.
\end{proof}

\begin{lem} \label{lem: main}
Let $k$ be a power of a prime, $A=\F_k$, $G$ be a connected $n$-vertex simple graph with $m$ edges, $T$ be a spanning tree of $G$, and  $D$ be the canonical orientation of $G$. Suppose that $(D, \sigma)$ is an $S_{A}$-labeling of $G$ with the property that $\sigma(e)$ is the identity permutation whenever $e \in E(T)$. If there is a proper $S_{A}$-$k$-coloring of $(D, \sigma)$, then there exists a polynomial $f_{(D, \sigma)} \in \mathcal{F}_{G,(D,\sigma)}$  of degree at most $(k-2)(m-n+1)+n-1$. 
\end{lem}

\begin{proof} 
Let $V(G)=\{v_1,v_2,\ldots v_n\}$. Suppose that $\kappa:V(G)\rightarrow A$ is a proper $S_{A}$-$k$-coloring of $(D, \sigma)$. We will now define a polynomial for each edge. Whenever $v_iv_j \in E(T)$ with $j > i$, let $f_{\sigma(v_iv_j)}=x_i-x_j$. Note that for these edges $f_{\sigma(v_iv_j)}(a,\sigma(v_iv_j)(a))=f_{\sigma(v_iv_j)}(a,a)=0$ for any $a\in A$. Also, $f_{\sigma(v_iv_j)}(\kappa(v_i),\kappa(v_j))\neq 0$ because $\kappa(v_i)\neq \kappa(v_j)$ since $\kappa$ is a proper $S_{A}$-$k$-coloring and $\sigma(v_iv_j)$ is the identity. 

For all other edges $v_iv_j \in E(G)-E(T)$ with $j>i$,  Proposition~\ref{prop:f_pi_for_simple} gives us a function $f_{\sigma(v_iv_j)} \in \F_k[x_i,x_j]$ that has $f_{\sigma(v_iv_j)}(a,\sigma(v_iv_j)(a))=0$ for any $a\in \F_k$ and $f_{\sigma(v_iv_j)}(\kappa(v_i),\kappa(v_j))\neq 0$. We are also guaranteed $f_{\sigma(v_iv_j)}$ has degree $d=k-2$.  

Now let
$$f_{(D,\sigma)}=\prod_{(v_i,v_j)\in E(D)}f_{\sigma(v_iv_j)}.$$
Note that by construction $f_{(D, \sigma)}(\kappa(v_1), \ldots, \kappa(v_n))\neq 0$.  Moreover, for any $g: V(G) \rightarrow A$ if $f_{(D, \sigma)}(g(v_1), \ldots, g(v_n))\neq 0$, then $f_{\sigma(v_iv_j)}(g(v_i),g(v_j))\neq 0$ for all edges $v_iv_j \in E(G)$ with $j > i$ which means that $g$ is a proper $S_{A}$-$k$-coloring of $(D, \sigma)$.  Consequently, $f_{(D, \sigma)} \in \mathcal{F}_{G,(D,\sigma)}$.
\end{proof}

The degree $d=k-2$ is higher than we would like because the condition $nk\geq d(m-n+1)+n-1$ will only apply to sparse graphs. However, there are certain permutations, linear permutations, that we focus on in Section~\ref{sec: linear}, which only require a degree one polynomial. Now, Corollary~\ref{thm: bound}, Lemma~\ref{lem: dpchromaticnumber}, and Lemma~\ref{lem: main} imply Theorem~\ref{thm:using d=k-2 result} which we restate.

\begin{customthm}{\bf \ref{thm:using d=k-2 result}}
Let $k=p^r$ where $p$ is prime, $r \in \N$, and $k > 2$. Suppose $G$ is a connected $n$-vertex simple graph with $m$ edges, and $T$ is a spanning tree of $G$.  Then the following statements hold. 
\\
(i)  Let $q = \lfloor k/2 \rfloor$. Let $G'$ be the multigraph obtained from $G$ by adding $(q - 1)$ parallel edges to each edge $e \in E(G) - E(T)$.  If $\chi_{DP}(G') \leq k$ and $m \leq n(1 + (k-2)/q) -1 + 1/q$, then
$$ P_{DP}(G, k) \geq k^{(n(q+k-2)-qm+1-q)/(k-1)}.$$
\\
(ii)  If $\chi_{DP}(G)\leq k$ and
$m\leq 2n-\frac{k-3}{k-2}$, then
$$
P_{DP}(G, k) \geq k^{((2n-m)(k-2) -(k-3))/(k-1)}.
$$
\end{customthm}

\begin{rmk}
\label{rmk:degree restrictions}
The degree of $k-2$ that we prove in Proposition~\ref{prop:f_pi_for_simple} can not be improved upon in that setup. This is  because of permutations like $\pi = 10234\cdots(p-1)$ of $\mathbb{F}_p$ for prime $p$ together with the ordered pair $(0,0)$ where we want $f_{\pi}(0,0)\neq 0$. This was computationally confirmed for primes $p$ up to $53$. Because this degree is rather high we can only apply Corollary~\ref{thm: bound} to $P_{DP}(G,k)$ for sparse graphs. Particularly our bound for an $n$-vertex, $m$-edge graph $G$ is $m\leq 2n-\frac{k-3}{k-2}$ as presented in Theorem~\ref{thm:using d=k-2 result}(ii).  %Additionally we see that the lower bound we prove for $P_{DP}(G,k)$ in Theorem~\ref{thm:using d=k-2 result}(ii) worsens as $k$ gets larger. See Figure~\ref{fig:bound tightness analysis}. 

In contrast we see in Proposition~\ref{prop: edgecoloring}(i) the best degree associated to a general single edge is
$q = \lfloor k/2 \rfloor$, which is smaller than $k-2$. This allows us to work with a larger family of graphs since the edge bound becomes
$m \leq n(1 + (k-2)/q) -1 + 1/q$. However, the drawback for Theorem~\ref{thm:using d=k-2 result}(i) is that we are working with multigraphs. The multigraph $G'$ is obtained from $G$ by adding $(q - 1)$ parallel edges to each edge outisde a spanning tree. Adding multiedges is likely to increase $\chi_{DP}(G')$ as compared to $\chi_{DP}(G)$, so the theorem only applies for larger $k$-values. How much $\chi_{DP}(G)$ differs from $\chi_{DP}(G')$ has not been studied. 

This is not to say these results could not be improved upon. For example, in Theorem~\ref{thm:using d=k-2 result}(ii), the way we constructed a polynomial to cover an $S$-labeling of $G$ was as a multiplication over the edges, $f=\prod_{e\in E(G)}f_e$. The restriction of degree $k-2$ per edge from Proposition~\ref{prop:f_pi_for_simple} comes from how we are gluing together $f_e$'s while maintaining a global nonzero output of $f$ for some existing DP-coloring. It might be possible to construct $f$ from  larger subgraphs of $G$ than just those from individual edges while maintaining a global nonzero output of $f$ for some existing DP-coloring. 
\end{rmk}

In Theorem~\ref{thm:using d=k-2 result}(ii) we require that the number of colors $k>2$ be a power of a prime. 
We can use Proposition~\ref{prop:f_pi_for_simple} together with Corollary~\ref{thm: bound} very similarly as we did in Theorem~\ref{thm:using d=k-2 result}(ii) for a general $c>2$ number of colors. 
Pick a power of a prime $k\geq c$ and when applying Corollary~\ref{thm: bound} use $A_i=B$ for all $i$ for some fixed $B\subseteq \mathbb{F}_k$ of cardinality $c$. 
%and will construct a polynomial over $\mathbb{F}_k$ exactly as in Lemma~\ref{lem: main}, but we restrict our attention to 
Instead of considering $k$-$S_{\mathbb{F}_k}$-colorings of $G$, we are considering $c$-$S_{\mathbb{F}_k}$-colorings. 
In other words, colorings of $S$-labelings for $S\subseteq S_{\mathbb{F}_k}$ where $S$ contains all permutations of $\mathbb{F}_k$ where elements of $\mathbb{F}_k-B$ are fixed points. 
We  follow the same argument as in Theorem~\ref{thm:using d=k-2 result}(ii) by constructing a polynomial like in Lemma~\ref{lem: main} and then apply Corollary~\ref{thm: bound}. 

%In Figure~\ref{fig:lower bound general c} we illustrate that sometimes Corollary~\ref{cor:bound for general c} gives us a better lower bound for $P_{DP}(G,c)$, but other times Theorem~\ref{thm:using d=k-2 result} gives us a better lower bound for  $P_{DP}(G,c)\geq P_{DP}(G,k)$ using a power of a prime $k$ where $k\leq c$. 
%\tre{(Above paragraph modified to describe argument for Corollary 20)}

\begin{cor}
\label{cor:bound for general c}
Let  $c\geq 2$ and $k=p^r$ be a power of prime with $c\leq k$ and $k>2$. Suppose $G$ is a connected $n$-vertex simple graph with $m$ edges.  If  $\chi_{DP}(G)\leq c$ and
$m \leq \frac{n(c+k-4)}{k-2} -\frac{k-3}{k-2}$, then
$$
P_{DP}(G, c) \geq c^{(n(c+k-4)-(k-2)m-(k-3))/(c-1)}.
$$
\end{cor}

It can be seen that sometimes Corollary~\ref{cor:bound for general c} gives us a better lower bound for $P_{DP}(G,c)$, but other times Theorem~\ref{thm:using d=k-2 result} gives us a better lower bound for  $P_{DP}(G,c)\geq P_{DP}(G,k)$ using a power of a prime $k$ where $k\leq c$.

\section{S-labelings with Linear Permutations} \label{sec: linear}

In Proposition~\ref{prop:f_pi_for_simple} we proved that for $k=p^r$, a power of a prime, we may need degree at worst $k-2$ to cover an edge labeled with permutation $\pi$ in the $S$-labeling. We discuss in Remark~\ref{rmk:degree restrictions} the drawbacks. If this degree of $k-2$ could be lowered we would be able to apply our theory to more graphs and our lower bound would be better. In this section we restrict our attention to using $S$-labelings where $S$ contains permutations $\pi$ that can be covered by a degree one polynomial. We will now formally define what we mean. 

We say a polynomial $f\in \F_k[x,y]$ {\it covers} a permutation $\pi:\F_k\rightarrow\F_k$ if $f_{\pi}(x,\pi(x))=0$ for all $x\in\F_k$. We will call a permutation {\it $\F_k$-linear} if there exists a degree one polynomial in $f\in \F_k[x,y]$ that covers $\pi$. 

\begin{pro}
\label{prop:linear_permutations_condition}
For any power of a prime  $k=p^r$ a permutation $\pi:\F_k\rightarrow \F_k$ is $\F_k$-linear if and only if there exist  an $a,b\in \F_k$ where $a\neq 0$ and   $\pi(x)=ax+b$. 
\end{pro}
\begin{proof}
First, suppose we have $a,b\in \F_k$ where $a\neq 0$. Note that the function $\pi:\F_k\rightarrow\F_k$ defined by $\pi(x)=ax+b$ is a permutation since there is an inverse function $\pi^{-1}:\F_k\rightarrow\F_k$ defined by $\pi^{-1}(x)=a^{-1}(x-b)$. This permutation $\pi$ is $\F_k$-linear because the degree one polynomial $$f_\pi(x,y)=x-a^{-1}y+a^{-1}b$$ covers $\pi$. We can see that 
$f_\pi(x,\pi(x))=x-a^{-1}(ax+b)+a^{-1}b=0$ for all $x\in\mathbb{F}_k$.

Now suppose that $\pi$ is a $\F_k$-linear permutation that is covered by the degree one polynomial $f(x,y)=ax+by+c$ for some $a,b,c\in\F_k$. First note that $a\neq 0$ and $b\neq 0$. Because if $a=b=0$, then $f$ is not degree one. If $a\neq 0$ but $b=0$, then $f(x,\pi(x))=ax+c$ can only be zero for one $x$-value contradicting the assumption that $f$ covers $\pi$. If $b\neq 0$ but $a=0$, then $f(x,\pi(x))=a\pi(x)+c=0$ for all $x$. This means that $\pi$ is a constant function, another contradiction. We now claim that the permutation $\pi$ must be  $\pi(x)=ab^{-1}x-b^{-1}c$, which satisfies all the conditions needed. This follows because we know that $f(x,\pi(x))=ax+b\pi(x)+c=0$ which implies that $\pi(x)=-b^{-1}(ax+c)$.
\end{proof}

\begin{cor}
For any power of a prime $k$, there are $k(k-1)$ $\F_k$-linear permutations. 
\end{cor}
\begin{proof}
This follows quickly from Proposition~\ref{prop:linear_permutations_condition} with there being a $\F_k$-linear permutation for each pair $a,b\in\F_k$ with $a\neq 0$. 
\end{proof}

For any $k$, a power of a prime, let $\mathcal{L}_k$ be the collection of linear permutations $\F_k\rightarrow\F_k$. We say a graph  $G$ is \emph{linearly $k$-colorable} if $G$ is $\mathcal{L}_k$-$k$-colorable.   The \emph{linear-chromatic number} of $G$, denoted $\chi_{\mathcal{L}}(G)$, is the smallest $m$ such that $G$ is linearly $m$-colorable.    The \emph{linear color function} of $G$ evaluated at $m$, denoted $P_{\mathcal{L}}(G,m)$, is given by $P_{\mathcal{L}}(G,m)=P_{\mathcal{L}_m}(G,m) $ and is equal the the minimum number of  colorings over all $\mathcal{L}_m$-labelings. Note that, for any $k \in \N$, $$P_{DP}(G, k) \leq P_{\mathcal{L}}(G,k) \leq P(G,k).$$ 

\begin{rmk}
We will note that these definitions are only defined for powers of primes and not for all positive integers. This is because our methods rely on us working over a field. While our definition of linear permutation is easily generalizable to any ring we have not studied or defined the notation of $P_{\mathcal{L}}(G,k)$ when $k$ is not a power of a prime. We lastly note that  there exist generalizations of Theorem~\ref{thm: AandF} to polynomials over rings with certain conditions~\cite{BCPS18}, although, it is not clear how we could apply these generalizations to the study of colorings of $S$-labelings in general.  
%might be weakened to working over an integral domains, we note that there does not exist a ring of order six that is also an integral domain. It is open what the meaning should be for linearly $k$-colorable for non-powers of primes $k$.
\end{rmk}

Similar to Theorem~\ref{thm:using d=k-2 result} we can construct a polynomial  like we did in Lemma~\ref{lem: main} using the polynomials described in Proposition~\ref{prop:linear_permutations_condition}. We then can  apply Corollary~\ref{thm: bound}  to get Theorem~\ref{thm: linear}, which we restate here. 

\begin{customthm}{\bf \ref{thm: linear}}
Let $k$ be a power of a prime. If an $n$-vertex graph $G$ with $m$ edges is $S$-$k$-colorable for some $S\subseteq\mathcal{L}_k$ and $m\leq (k-1)n$, then
$$P_S(G,k)\geq k^{n-\frac{m}{k-1}}$$
and in particular, $P_{\mathcal{L}}(G,k)\geq k^{n-\frac{m}{k-1}}.$
\end{customthm}

%\begin{cor} NOT NEEDED
%Let $k$ be a power of a prime. If an $n$-vertex graph $G$ with $m$ edges is linearly $k$-colorable and $m\leq (k-1)n$, then
%$$P_{\mathcal{L}}(G,k)\geq k^{(kn-n-m)/(k-1)}=k^{n-\frac{m}{k-1}}.$$
%\end{cor}

We conjecture that the linear color function and the DP color function are equal for power of primes $k$. This is trivially true for $k=3$ and has been confirmed by computer for $k\leq 4$ and $n\leq 5$. 

\begin{customconj}{\bf \ref{conj: probablyfalse}}
For $k$, a power of a prime,
$$P_{\mathcal{L}}(G,k)=P_{DP}(G,k).$$
\end{customconj}

%%%%%%%%%%%%%%%%%%%%%%%%%%
\subsection{Colorings of Signed Graphs}
\label{subsec:signed}
One application of $\mathcal{L}_k$-colorings is to colorings of signed graphs.  Signed graphs and their colorings were formally introduced by  Zaslavsky~\cite{Z82} in 1982 (although versions of it have been studied since 1940s), and have been widely studied since then (see a recent survey in \cite{SV21}). 
We will introduce a notation for signed graphs and their coloring that is similar to notation used in Steffen and Vogel~\cite{SV21}.
%M\'{a}\v{c}ajov\'{a},  Raspaud and \v{S}koviera~\cite{MRS}. 
%This modification is so we can more easily talk about colorings of signed graphs and $S$-labelings at the same time. 

A signed graph  $(G,\epsilon)$ is a graph $G$ together with a map $\epsilon:E(G)\rightarrow \{-1,1\}$ that assigns to each edge a sign,  positive or negative.  
The traditional color sets for signed graphs are $M_{2t+1}=\{0,\pm 1,\pm2,\ldots,\pm t\}$ and $M_{2t}=\{\pm 1,\pm2,\ldots,\pm t\}$. 
A {\it coloring of a signed graph} is a map $\kappa:V(G)\rightarrow M_k$ where for any edge $e=uv$ we have that $\kappa(v)\neq\epsilon(e)\kappa(u)$. 
We call these $k$-colorings of $(G,\epsilon)$ and if a $k$-coloring exists we call the signed graph {\it $k$-colorable}. 
%A coloring is {\it zero-free} if zero is not used in the coloring. This means that all $(2t)$-colorings are zero free and only some $(2t+1)$-colorings are zero-free. This distinction emanates from the distinct role zero plays in coloring. 
The {\it chromatic number} $\chi(G,\epsilon)$ of a signed graph $(G,\epsilon)$ is the smallest $k$ such that $(G,\epsilon)$ is $k$-colorable. The {\it chromatic function} of $(G,\epsilon)$ is $P((G,\epsilon),k)$, which is equal to the number of $k$-colorings of $(G,\epsilon)$. 
The chromatic number of a signed graph $(G,\epsilon)$ relates to the chromatic number of its underlying graph $G$ by $\chi(G,\epsilon)\leq 2\chi(G)-1$, a bound that was proven to be tight in~\cite{MRS16}. 

%\tre{(Sam note: New definitions added.) (Hemanshu: Since we use $S$ for $S$-labeling, would it be better to use $\pm$ as subscript for all these?)}
Since most of our results can apply to colorings of signed graphs regardless of the signs chosen on the edges we define the following. The {\it signed chromatic number} $\chi_{\pm}(G)$ of a  graph $G$ is the smallest $k$ such that $(G,\epsilon)$ is $k$-colorable for any choice of $\epsilon$. The {\it signed chromatic function} of a graph $G$, $P_{\pm}(G,k)$, is the minimum $P((G,\epsilon),k)$ over all choices of signed $\epsilon$. Both the signed chromatic number and function sit inbetween the DP chromatic number and function and the classical chromatic number and function, $$\chi(G)\leq \chi_{\pm}(G)\leq \chi_{DP}(G) \text{, and } P_{DP}(G,k)\leq P_{\pm}(G,k)\leq P(G,k).$$
We will now reframe  colorings of a signed graph $(G,\epsilon)$ as colorings of some $S$-labeling of $G$.

\begin{pro}
\label{prop: sign-colorings1}
Given any signed graph $(G,\epsilon)$ with coloring $\kappa:V(G)\rightarrow M_{2t+1}$ there is an associated coloring of an $S$-labeling of $G$ with $A=\mathbb{Z}_{2t+1}$ and $S=\{\text{id},\pi\}$ where $\pi:A\rightarrow A$ sends $a\mapsto -a$.

Similarly, given a signed graph $(G,\epsilon)$ with a zero-free coloring $\kappa:V(G)\rightarrow M_{2t}$ there is an associated coloring of an $S$-labeling of the underlying graph of $G$ with $A=\mathbb{Z}_{2t+1}-\{0\}$ and $S=\{\text{id},\pi\}$ where $\pi:A\rightarrow A$ sends $a\mapsto -a$.
\end{pro}

\begin{proof}
Suppose $(G,\epsilon)$ is a signed graph with coloring $\kappa:V(G)\rightarrow M_{2t+1}$. Let $D$ be any digraph of $G$, $A=\mathbb{Z}_{2t+1}$, and $S=\{\text{id},\pi\}$ where $\pi:A\rightarrow A$ sends $a\mapsto -a$. The associated $S$-labeling $(D,\sigma)$ of $G$ has $\sigma:V(G)\rightarrow S$ defined as $\sigma(e)=\text{id}$ if $\epsilon(e)=1$ and $\sigma(e)=\pi$ if $\epsilon(e)=-1$. The coloring $\kappa$ is also a map $\kappa:V(G)\rightarrow \mathbb{Z}_{2t+1}$. Because $\kappa$ was a coloring of the signed graph $(G,\epsilon)$, we know that $\kappa$ is a coloring of $(G,\sigma)$ because $\kappa(v)\neq\epsilon(e)\kappa(u)$ is equivalent to $\kappa(v)\neq\sigma(e)(\kappa(u))$. 

The second part follows because if we replace $M_{2t+1}$ with $M_{2t}$ and  $A=\mathbb{Z}_{2t+1}$ with $A=\mathbb{Z}_{2t+1}-\{0\}$, then the  above argument still holds.
\end{proof}

To use our methods in this paper we need $A$ to be a finite field, which is not true in general in  the above proposition. 
The following corollary of Proposition~\ref{prop: sign-colorings1} completes the reframing to when $A$ is a finite field. 

\begin{cor}
\label{cor:signed are linear}
Let $k$ be a power of an odd prime. 
Given any signed graph $(G,\epsilon)$ with coloring $\kappa:V(G)\rightarrow M_{k}$ there is an associated coloring of an $S$-labeling of $G$ with $A=\mathbb{F}_{k}$ and $S=\{\text{id},\pi\}$ where $\pi:A\rightarrow A$ sends $a\mapsto -a$.

Let $k=2^r$. Given a signed graph $(G,\epsilon)$ with a zero-free coloring $\kappa:V(G)\rightarrow M_{k}$ there is an associated coloring of an $S$-labeling of  $G$ with $A=\mathbb{F}_{k}$ and $S=\{\text{id},\pi\}$ where $\pi:A\rightarrow A$ sends $a\mapsto a+1$.
\end{cor}

\begin{proof}
The first part of this proof follows almost directly from Proposition~\ref{prop: sign-colorings1}, but due to the structural differences between $\Z_k$ and $\F_k$ there is some work to do.
Let $k$ be a power of an odd prime, $(G,\epsilon)$ be a signed graph with a $k$-coloring $\kappa:V(G)\rightarrow M_{k}$, and $D$ be any digraph of $G$.  By Proposition~\ref{prop: sign-colorings1} we know there is an associated coloring of an $S$-labeling $(D,\sigma)$ of  $G$ with $A=\mathbb{Z}_{k}$ and $S=\{\text{id},\pi\}$ where $\pi:A\rightarrow A$ sends $a\mapsto -a$ where the associated coloring is $\kappa$ viewed as $\kappa:V(G)\rightarrow A$. 
Let $A'=\F_k$, $S'=\{\text{id},\pi'\}$ where $\pi':A'\rightarrow A'$ sends $a\mapsto -a$ and $\sigma':E(G)\rightarrow S'$ be the same as $\sigma$. We mean let $\sigma'(e)=\pi'$ if $\sigma(e)=\pi$ and $\sigma'(e)=\text{id}$ if $\sigma(e)=\text{id}$.   
Note that $\kappa$ does not generalize to a coloring with color set $\F_k$. 
Let $\psi:A\rightarrow\F_k$ be a bijection between these two sets of equal cardinality such that $\psi(-a)=-\psi(a)$. Such a bijection exists since additive inverse of $\Z_k$ and $\F_k$ come in pairs expect for zero. Consider the coloring $\kappa':V(G)\rightarrow \F_k$ defined by $\kappa'(v)=\psi\circ\kappa(v)$. This is the coloring of the $S'$labeling of $(D,\sigma)$ since for any edge $uv$ we have that $\kappa(u)\neq \sigma(uv)(\kappa(v))$ if and only if $\kappa'(u)\neq \sigma'(uv)(\kappa'(v))$.

The second part will need a little more work. By proposition Proposition~\ref{prop: sign-colorings1} we have the following. Let $k=2^r$ be a power of an even prime and $(G,\epsilon)$ be a signed graph with a zero-free coloring $\kappa:V(G)\rightarrow M_{k}$. We know there is an associated coloring of an $S$-labeling $(D,\sigma)$ of  $G$ with $A=\mathbb{Z}_{k+1}-\{0\}$ and $S=\{\text{id},\pi\}$ where $\pi:A\rightarrow A$ sends $a\mapsto -a$ where the associated coloring is $\kappa$ viewed as $\kappa:V(G)\rightarrow A$. 

Now we want to translate everything to $A'=\F_k$. Let $\psi:A\rightarrow\F_k$ be any bijection between these two sets of equal cardinality. Let $\pi':\F_k\rightarrow\F_k$ be the associated permutation defined by $\psi$ from $\pi$, $a\mapsto \psi\circ\pi\circ\psi^{-1}(a)$. Let  $S'=\{\text{id},\pi'\}$ and the $S'$-labeling $(D,\sigma')$ be defined similarly to $(D,\sigma)$ that $\sigma':V(G)\rightarrow S$ is defined as $\sigma'(e)=\text{id}$ if $\epsilon(e)=1$ and $\sigma'(e)=\pi'$ if $\epsilon(e)=-1$. 
Now let us define a coloring of $(D,\sigma')$. Define $\kappa':V(G)\rightarrow \mathbb{F}_k$ as $\kappa'(v)=\psi\circ\kappa(v)$. Because $\kappa$ was a coloring of the signed graph $(G,\epsilon)$, we know that $\kappa'$ is a coloring of $(D,\sigma')$ because $\kappa(v)\neq \sigma(e)(\kappa(u))$ for an edge $e=uv$ is equivalent to $\kappa'(v)\neq\sigma'(e)(\kappa'(u))$. 

Consider the permutation $\tau:\F_k\rightarrow\F_k$ where $a\mapsto a+1$. Recall that $A=\F_k$ can be viewed as $\mathbb{F}_2[x]/(p(x))$ where $p(x)\in\mathbb{F}_2[x]$ is an irreducible polynomial of degree $r$. This shows that the  permutation $\tau$ has $k/2$ two-cycles just like the permutation $\pi'$ and $\pi$ that all have the same cycle type. By Lemma~\ref{lem:equivalent S-labeling}(ii) we are done because $\pi'$ and $\tau$ are in the same conjugacy class. 
\end{proof}

A lot of the work we had to do was to not just translate signed graphs into $S$-labelings, but to  translate colorings of signed graphs into colorings of $S$-labelings where $S$ contains only linear permutations. This is what we have completed in  the corollary above. Now we can apply our methods from this paper.

% \begin{customthm} {\bf \ref{thm:signed lower bound}}
% Let $(G,\epsilon)$ be an $n$-vertex signed graph with $m$ edges. 
% Let $k=p^r$ be a power of a prime.  If   $\chi(G,\epsilon)\leq k$ and $m\leq (k-1)n$, then
% $$P((G,\epsilon),k)\geq k^{(kn-n-m)/(k-1)}=k^{n-\frac{m}{k-1}}.$$
% \end{customthm}
\begin{thm} 
\label{thm:signed lower bound}
Let $(G,\epsilon)$ be an $n$-vertex signed graph with $m$ edges. 
Let $k=p^r$ be a power of a prime.  If   $\chi(G,\epsilon)\leq k$ and $m\leq (k-1)n$, then
$$P((G,\epsilon),k)\geq k^{n-\frac{m}{k-1}}.$$
\end{thm}
\begin{proof} Let $(G,\epsilon)$ be an $n$-vertex signed graph with $m$ edges, $\chi(G,\epsilon)\leq k$,  $m\leq (k-1)n$ and $k$ be a power of a prime. By Corollary~\ref{cor:signed are linear} the colorings of the signed graph $(G,\epsilon)$ are in bijection colorings of an $S$-labeling where $S=\{\text{id}, \pi\}$ for some $\pi$ determined by the pairity of $k$.
Note that for $k$, a power of an odd prime, that $\pi:\F_k\rightarrow \F_k$ that sends $a\mapsto -a$ is a $\F_k$-linear permutation. Also, note that for $k$, a power of an even prime, that $\pi:\F_k\rightarrow \F_k$ that sends $a\mapsto a+1$ is a $\F_k$-linear permutation since  $\F_k$ can be viewed as $\mathbb{F}_2[x]/(p(x))$ where $p(x)\in\mathbb{F}_2[x]$ is an irreducible polynomial in $\F_2[x]$. 
This makes $S\subseteq \mathcal{L}_k$ whether $k$ is even or odd. By Theorem~\ref{thm: linear} we have our result. %\tre{(Sam Note: We aren't actually needing $S$-$k$-colorable, but just $(D,\sigma)$ is $k$-colorable, unless I am mixing something up. This makes the result a bit stronger.)}
\end{proof}

We distinguish Theorem~\ref{thm:signed lower bound} from its corollary below because Theorem~\ref{thm:signed lower bound} might apply to a signed graph $(G,\epsilon)$ for much smaller $k$ as compared to its underlying graph $G$. This is because $\chi((G,\epsilon),k)$ might be far from  $\chi_{\pm}(G,k)$. For example, if all edges in $(G,\epsilon)$ are negative then $\chi((G,\epsilon),k)= 2$, so for these signed graphs Theorem~\ref{thm:signed lower bound} applies starting at $k\geq 2$.

\begin{customcor}{\bf \ref{cor:signed lower bound}}
Let $G$ be an $n$-vertex signed graph with $m$ edges. 
Let $k$ be a power of a prime.  If  $\chi_{\pm}(G)\leq k$ and $m\leq (k-1)n$, then 
$$P_{\pm}(G,k)\geq k^{n-\frac{m}{k-1}}.$$
\end{customcor}

%%%%%%%%%%%%%%%%%%%%%%%%%%%%%%%%%
\section{Exponentially Many Colorings  of Planar Graphs}\label{planar}

The history of coloring of planar graphs and its subfamilies, is intertwined with the related conjectures on existence of exponentially many such colorings \footnote{In this section, $n$ will always denote the number of vertices of the graph under consideration, and when we say exponential, we mean exponential in $n$.}. In 1946, giving an enumerative extension of the four color conjecture, Birkhoff and Lewis~\cite{BL46} conjectured that for any $n$-vertex planar graph $G$, $P(G,k) \ge k(k - 1)(k-2)(k-3)^{n-3}$ for all real numbers $k \ge 4$ \footnote{Note that Birkhoff and Lewis Conjecture does not claim there are exponentially many 4-colorings of planar graphs.}. They proved this is true for $k\ge 5$ which gives exponentially many 5-colorings of planar graphs (see~\cite{T06} for a simpler proof). Appel and Haken~\cite{AH77, AH772} famously proved the Four Color Theorem in 1976. After Thomassen~\cite{T94} in 1994 proved all planar graphs are 5-choosable, it was asked whether there exponentially many 5-list-colorings of planar graphs. In 2007, Thomassen~\cite{T07b} proved that this is true by showing $P_{\ell}(G,5) \ge 2^{n/9}$ when $G$ is a planar graph, and correspondingly for DP-colorings, Postle and Smith-Roberge~\cite{PS23} recently showed that $P_{DP}(G,5) \ge 2^{n/67}$. Recall that for any $k \in \N$, $P_{DP}(G, k) \leq P_\ell(G,k) \leq P(G,k)$, and $P_{DP}(G, k) \leq P_{\pm}(G,k) \leq P(G,k)$. Below, using Corollary~\ref{cor:signed lower bound}, we will show that $P_{\pm}(G,5)$ is exponential, with a better bound, for all planar graphs $G$, extending the sharp bound $\chi_{\pm}(G) \le 5$ from 2016~\cite{MRS16}.

The question of number of colors needed for sparse planar graphs, where sparsity is controlled by forbidding short cycles, also has a long history. For planar graphs of girth 5, Thomassen~\cite{T95} proved that they are 3-choosable, and $P_{\ell}(G,3) \ge 2^{n/1000}$~\cite{T07} for any such graph $G$. This was recently improved by Postle and Smith-Roberge~\cite{PS23} to $P_{DP}(G,3) \ge 2^{n/282}$. All of these results use intricate and long topological arguments. In a previous paper~\cite{DKM22}, we improved this further to $P_{DP}(G,3) \ge 3^{n/6}$. Now, using the technique outlined in Theorem~\ref{thm:using d=k-2 result}, we can show there are exponentially many such DP-colorings for many such families of sparse planar graphs. 

In principle, we can guarantee exponentially many DP-colorings for any family of $n$-vertex graphs whose DP-chromatic number is known and whose number of edges is bounded above by $(2-c)n$ for some $c>0$. This required bound on number of edges can be weakened as described in Theorem~\ref{thm:using d=k-2 result}(i) if we know the DP-chromatic number of a certain multigraph containing $G$, which is a topic of future research. More dramatically, this bound can be weakened to $(k-1)n$ in case of $k$-colorings of signed graphs as shown in Theorem~\ref{thm: linear}.  In the rest of this section, we illustrate some of these applications.

We start with an easy application to embedded graphs whose number of edges can be bounded directly using Euler's formula.

\begin{thm}\label{thm: Planar1}
   Let $G$ be an $n$-vertex graph of girth at least 5 embedded on a surface of Euler genus $g$. Suppose $G$ is DP-$k$-colorable where $k$ is a power of a prime and $k > 2$. If $n \ge 5g$, then $P_{DP}(G,k)\geq k^{(((n-5g)(k-2)/3) - (k-3))/(k-1)}$.  
\end{thm}

By Euler's formula, the number of edges in such a graph is bounded by $(5n-10+5g)/3$. This edge bound together with Theorem~\ref{thm:using d=k-2 result}(ii) proves the statement above.

Theorem~\ref{thm: Planar1} generalizes a 2007 result of Thomassen~\cite{T07} that such graphs have $P(G,3) \ge 2^{(n-5g)/9}$. For a planar graph of girth at least 5, $G$, we know $\chi_{DP}(G) \le 3$, and Theorem~\ref{thm: Planar1} reduces to $P_{DP}(G,3) \ge 3^{n/6}$ which we proved in a previous paper~\cite{DKM22}.\\

Next, we consider families of sparse planar graphs where bounding the number of edges requires more than a simple application of Euler's formula. Recall that \emph{maximum average degree} of a graph $G$ is defined to be $mad(G) = \max\{2|E(H)|/|V(H)| : H \subseteq G\}$.  Since $mad(G) \ge 2|E(G)|/|V(G)|$, an upper bound on  $mad(G)$ gives us an upper bound on the number of edges of $G$. Using a standard discharging argument (see Exercise~3.21 in Cranston and West~\cite{CW17}), we get an upper bound $3 + 9/(2t-1)$ on the maximum average degree of planar graphs having no cycle of length in $\{4,...,t\}$, and then combining that bound with Euler's formula leads to the following fact.

\begin{fact}\label{fact: planar}
    Let G be an $n$-vertex plane graph having no cycle of length in $\{4,...,t\}$. Then, $|E(G)| < (n-2)(3t+3)/(2t-1)$. 
\end{fact}

We are now ready to prove the following exponential lower bounds on DP-colorings of sparse planar graphs.

\begin{thm}\label{thm: Planar2}
    Let $G$ be an $n$-vertex planar graph. The following statements are true whenever $k > 2$ and $k$ is a power of a prime.
    \begin{enumerate}[(i)]
    \item  If $G$ has no cycle of length in $\{4,5,6,7,8\}$, then $P_{DP}(G,k) \ge k^{((n/5)(k-2)/(k-1)) - 1}$ for all $k\ge \chi_{DP}(G)$~\footnote{It is only known that $\chi_{DP}(G) \in \{3,4\}$, see \cite{DP15}.}.
    
    \item If $G$ has no cycle of length in $\{4,5,6,9\}$, then $P_{DP}(G,k) \ge k^{((n/11)(k-2)/(k-1)) - 1} $ for all $k\ge 3$. In particular, $P_{DP}(G,3) \ge 3^{(n/22) - 1} $.
    
    \item If $G$ has no intersecting triangles and no cycle of length in $\{4,5,6,7\}$, then $P_{DP}(G,k) \ge k^{((2n/13)(k-2)/(k-1)) - 1}$ for all $k\ge 3$. In particular, $P_{DP}(G,3) \ge 3^{(n/13) - 1} $.
    
    \item If $G$ has no cycle of length in $\{4,5,6\}$, then $P_{DP}(G,k) \ge k^{((n/11)(k-2)/(k-1)) - 1} $ for all $k\ge 4$. In particular, $P_{DP}(G,4) \ge 3^{(2n/33) - 1} $.
    
    \item If $G$ has no cycle of length in $\{4,5,7,9\}$, then $P_{DP}(G,k) \ge k^{((2n/13)(k-2)/(k-1)) - 1} $ for all $k\ge 3$. In particular, $P_{DP}(G,3) \ge 3^{(n/22) - 1} $.
\end{enumerate}
\end{thm}

\begin{proof}
Each of these results follows from Theorem~\ref{thm:using d=k-2 result}(ii) using the following observations.

(i) Applying Fact~\ref{fact: planar} with $t=8$ to $G$ gives $|E(G)| < (9/5) n$.

(ii) It is known that $\chi_{DP}(G) \le 3$ (\cite{LLYY19}). Applying Fact~\ref{fact: planar} with $t=6$ to $G$ gives $|E(G)| < (21/11) n$. 

(iii) It is known that $\chi_{DP}(G) \le 3$ (\cite{L22}). Applying Fact~\ref{fact: planar} with $t=7$ to $G$ gives $|E(G)| < (24/13) n$. 

(iv) It is known that $\chi_{DP}(G) \le 4$ (\cite{KO18}). Applying Fact~\ref{fact: planar} with $t=6$ to $G$ gives $|E(G)| < (21/11) n$. 

(v) It is known that $\chi_{DP}(G) \le 3$ (\cite{LLYY19}). Applying Fact~\ref{fact: planar} with $t=5$ to $G$ gives $|E(G)| < 2n$, which is too weak for our purpose. However, it is possible to give a short discharging argument~\footnote{Dan Cranston, personal communication, 2023.} that shows that $G$ has average face length at least $4+(4/11)$ which implies that $|E(G)| < (24/13) n$.
\end{proof}

Since $P_{DP}(G, k) \leq P_{\pm}(G,k) \leq P(G,k)$ and $\chi(G) \leq \chi_{\pm}(G) \leq \chi_{DP}(G)$, all the bounds above also apply to the number of colorings of signed planar graphs. However, the bound on the number of edges required in Corollary~\ref{cor:signed lower bound} is much weaker than that in Theorem~\ref{thm:using d=k-2 result}(ii).  So, we can get better bounds on $P_{\pm}(G,k)$ by applying Corollary~\ref{cor:signed lower bound}.  From the result of Postle and Smith-Roberge~\cite{PS23}, $P_{DP}(G,5) \ge 2^{n/67}$. For any planar graph $G$, $\chi_{\pm}(G)\leq 5$ and this bound is tight (\cite{SV21}); so, $P_{\pm}(G,5) \ge 2^{n/67}$. We improve this bound using Cor~\ref{cor:signed lower bound} to $P_{\pm}(G,5)\geq 5^{\frac{n}{4}}$ for any planar graph $G$. In the case of planar graphs of girth 5, we get $P_{\pm}(G,3)\geq 3^{\frac{n}{6}}$ which matches the bound on $P_{DP}(G,3)$ from~\cite{DKM22}~\footnote{This is because all permutations in the field of order 3 are linear; so, we do not gain anything by focusing on only linear permutations.}

We show below there are exponentially many colorings of all signed triangle-free planar graphs, even though it is not known if there is such a result for number of DP-colorings of these graphs (see Section~\ref{triangle-free planar} below).

\begin{thm}\label{thm: Planar3}
Let $G$ be an $n$-vertex planar graph. The following statements are true whenever $k>2$ and $k$ is a power of a prime.
\begin{enumerate}[(i)]

\item
We have $P_{\pm}(G,k)\geq k^{\frac{n(k-4)}{k-1}}$ for all $k\geq 5$. In particular,
$P_{\pm}(G,5)\geq 5^{\frac{n}{4}}$.
\item If $G$ is triangle free, then $P_{\pm}(G,k)\geq k^{\frac{n(k-3)}{k-1}}$ for all $k\geq 4$. In particular, 
$P_{\pm}(G,4)\geq 4^{\frac{n}{3}}$.
\item If the girth of $G$ is at least 5, then $P_{\pm}(G,k)\geq k^{\frac{n(3k-8)}{3(k-1)}}$ for all $k\geq 3$. In particular,
$P_{\pm}(G,3)\geq 3^{\frac{n}{6}}.$
\item If $G$ has no cycle of  length in $\{4,5,6,7,8\}$, then $P_{\pm}(G,k)\geq k^{\frac{n(5k-14)}{5(k-1)}}
$ for all $k\geq 3$. In particular,
$P_{\pm}(G,3)\geq 3^{\frac{n}{10}}.$
\end{enumerate}
\end{thm}

\begin{proof}
Each of these results follows from Corollary~\ref{cor:signed lower bound} using the following observations.

(i)
It is known that $\chi_{\pm}(G)\leq 5$ (\cite{MRS16,SV21}). Since $G$ is planar, it  follows from Euler's formula  that $|E(G)|\leq 3n-6$. 

(ii)
It is known that $\chi_{\pm}(G)\leq 4$ (\cite{MRS16}). Since $G$ is triangle free and planar, it  follows from Euler's formula  that $|E(G)|\leq 2n-4$. 

(iii)
It is known that $\chi_{\pm}(G)\leq 3$ (\cite{MRS16}). Since $G$ is planar with girth at least 5, it  follows from Euler's formula  that $|E(G)|\leq \frac{5}{3}(n-2)$. 

(iv)
It is known that $\chi_{\pm}(G) \leq 3$ (\cite{HL18}). Applying Fact\ref{fact: planar} with $t=8$ to $G$ gives $|E(G)| < \frac{9}{5}(n-2)$. 
\end{proof}

\subsection{Triangle-free Planar Graphs}\label{triangle-free planar}

Gr\"{o}tzsch's famous theorem~\cite{G59} states that triangle-free planar graphs are 3-colorable. However there exist triangle-free planar graphs are that are not 3-choosable~\cite{V95} and hence not DP-3-colorable. By a degeneracy argument, it follows that such graphs are 4-choosable and DP-4-colorable. 

Let $G$ be a triangle-free planar $n$-vertex graph. It was conjectured by Thomassen~\cite{T07} that $G$ has exponentially many 3-colorings. The best known result towards this is the subexponential bound $2^{\sqrt{n/212}}$ in~\cite{ADPT13}. However, it was recently shown that the conjecture is false (\cite{T23}) and the best we can hope for is $64^{n^{0.731}}$ 3-colorings (\cite{DP22}). So it's natural to ask whether there are exponentially many list-4-colorings, and more strongly, exponentially many DP-4-colorings of such graphs.

\begin{customconj}{\ref{conj: triangle-free planar}}
    There exists a constant $c >1$, such that for any triangle-free planar $n$-vertex graph $G$, $P_{DP}(G,4) \ge c^n$. 
\end{customconj}

Towards this conjecture, we show below that there are exponentially many list-4-colorings of triangle-free planar graphs with an easy application of Proposition~\ref{prop: listcoloring} (this was first shown in \cite{KP18} in context of triangle-free graphs embedded in a fixed surface of genus $g$; we can prove a similar result along the lines of Theorem~\ref{thm: Planar1}). We show this bound also holds for the number of colorings of signed triangle-free planar graphs using Corollary~\ref{cor:signed lower bound}. In case of DP-4-colorings, we provide some strong evidence towards Conjecture~\ref{conj: triangle-free planar}. Recall that the number of edges of a triangle-free planar $n$-vertex graph is at most $2n-4$. Our result shows that such a graph has exponentially many DP-4-colorings if its number of edges is at most $(2-c)n$ for some $c>1/2n$.

\begin{customthm} {\ref{thm: Planar4}}
Let $G$ be an $n$-vertex triangle-free planar graph with $m$ edges. The following statements are true.
\begin{enumerate}[(i)]
\item $P_{\ell}(G,4)\geq 4^{\frac{n+4}{3}}$.

\item $P_{\pm}(G,4)\geq 4^{\frac{n+4}{3}}$.

%\item Suppose $T$ is a spanning tree of $G$.  Let $G'$ be the multigraph obtained from $G$ by adding $1$ parallel edges to each edge $e \in E(G) - E(T)$.  If $\chi_{DP}(G') \leq 4$, then $P_{DP}(G, 4) \geq 4^{(n(4)-2(2n-4)-1)/(3)}.$ 

\item Suppose there exists $c>0$ such that $m \le (2-c)n$ and $cn \geq 1/2$, then $P_{DP}(G,4) \ge 4^{(4cn-1)/3}.$

\end{enumerate}

\end{customthm}

\begin{proof}
Recall $m \le 2n-4$ and $\chi_{\ell}(G)\leq 4$. We also know that $\chi_{\pm}(G)\leq 4$ (\cite{MRS16}). Now, parts (i) and (ii) follow from  Proposition~\ref{prop: listcoloring} and Corollary~\ref{cor:signed lower bound}, respectively.

Since $\chi_{DP}(G)\leq 4$, part (iii) follows from Theorem~\ref{thm:using d=k-2 result}(ii).
\end{proof}

\vspace*{1cm}
{\bf Acknowledgements.} The authors thank Dan Cranston for a very helpful conversation with regard to application of discharging arguments to bound the number of edges of sparse planar graphs as needed in Section~\ref{planar}.

\end{document}